\documentclass[11pt]{article}

\usepackage{amstext,amssymb,amsmath,amsbsy}
\usepackage{hyperref}
\usepackage{amscd}
\usepackage{amsfonts}
\usepackage{indentfirst}
\usepackage{verbatim}
\usepackage{amsmath}
\usepackage{amsthm}
\usepackage{enumerate}
\usepackage{graphicx}
\usepackage{color}
\usepackage[OT1]{fontenc}
\usepackage[latin1]{inputenc}
\usepackage[english]{babel}
\usepackage{amssymb}
\usepackage{subfig}
\usepackage{algorithm}
\usepackage{algpseudocode}

\newcommand{\R}{\mathbb{R}}

\newcommand{\ds}{\displaystyle}
\newcommand{\Id}{\textrm{Id}}

\newcommand{\x}{{\bf x}}

\newcommand{\bu}{{\bf u}}

\newcommand{\bh}{{\bf h}}

\newcommand{\e}{{\bf e}}
\newcommand{\ba}{{\bf a}}
\newcommand{\bb}{{\bf b}}
\newcommand{\bc}{{\bf c}}

\newcommand{\Div}{{\rm div}}

\newcommand{\uap}{{u^{\small N}}}
\newcommand{\uapt}{{u^{\small N}_t}}
\newcommand{\uaptt}{{u^{\small N}_{tt}}}

\renewcommand{\H}{{\bf H_0}}
\newcommand{\f}{{\bf f}}

\newtheorem{Theorem}{Theorem}[section]
\newtheorem{Lemma}{Lemma}[section]

\newtheorem{Proposition}{Proposition}[section]

\newtheorem{remark}{Remark}[section]

\newtheorem*{Assumption*}{Assumption}
\newtheorem{Definition}{Definition}[section]
\newtheorem{problem}{Problem}[section]
\newtheorem*{problem*}{Problem}
\numberwithin{equation}{section}

\begin{document}

\title{The quasi-reversibility method to numerically solve an inverse source problem for hyperbolic equations   }

\author{Thuy T. Le\thanks{
Department of Mathematics and Statistics, University of North Carolina at
Charlotte, Charlotte, NC 28223, USA, tle55@uncc.edu, loc.nguyen@uncc.edu, wpowel13@uncc.edu } \and Loc H. Nguyen\footnotemark[1] \and 
\and Thi-Phong Nguyen\thanks{Department of Mathematics, Purdue University, West Lafayette, IN, 47907, nguye686@purdue.edu (corresponding
author).}
\and William Powell\footnotemark[1]}

\date{}
\maketitle

\begin{abstract}
	We propose a numerical method to solve an inverse source problem of computing the initial condition of hyperbolic equations from the measurements of Cauchy data. This problem arises in thermo- and photo- acoustic tomography in a bounded cavity, in which the reflection of the wave makes the widely-used approaches, such as the time reversal method, not applicable. 
In order to solve this inverse source problem, we approximate the solution to the hyperbolic equation by its  Fourier series with respect to a special orthonormal basis of $L^2$. 
Then, we derive a coupled system of elliptic equations for the corresponding Fourier coefficients.
We solve it by the quasi-reversibility method. 
	The desired initial condition follows.
	We rigorously prove the convergence of the quasi-reversibility method as the noise level tends to 0. 
	Some numerical examples are provided. 
	In addition, we numerically prove that the use of the special basic above is significant.
\end{abstract}



\section{Introduction}

We consider an inverse source problem arising from biomedical imaging based on the photo-acoustic and thermo-acoustic effects, which are named as photo-acoustic tomography (PAT) and thermo-acoustic tomograpy (TAT) respectively.
In PAT, \cite{Krugeretal:mp1995, Oraevskyelal:ps1994}, non-ionizing laser pulses are sent to a biological tissue under inspection (for instance, woman's breast in mamography).
A part of the energy will be absorbed and converted into heat,
causing a thermal expansion and a subsequence ultrasonic wave propagating in space.
The ultrasonic pressures on a surface around the tissue are measured.
The experimental set up for TAT, \cite{Krugerelal:mp1999}, is similar to PAT except the use of microwave other than laser pulses.
Finding some initial information of the pressures from these measurements yields structure inside this tissue.

Due to the important real-world applications, the inverse source problem PAT/TAT has been studied intensively.  
There are several methods to solve them available. 
In the case when the waves propagate in the free space, one can find explicit reconstruction formulas in \cite{DoKunyansky:ip2018,  Haltmeier:cma2013, Natterer:ipi2012, Linh:ipi2009}, the time reversal method \cite{ KatsnelsonNguyen:aml2018, Hristova:ip2009, HristovaKuchmentLinh:ip2006, Stefanov:ip2009, Stefanov:ip2011}, the quasi-reversibility method \cite{ClasonKlibanov:sjsc2007} and the iterative methods \cite{Huangetal:IEEE2013, Paltaufetal:ip2007, Paltaufetal:osa2002}.
The publications above study PAT/TAT for simple models for non-damping and isotropic media. 
The reader can find publications about PAT/TAT for more complicated model involving a damping term or attenuation term \cite{Ammarielal:sp2012, Ammarietal:cm2011, Haltmeier:jmiv2019, Acosta:jde2018, Burgholzer:pspie2007, Homan:ipi2013, Kowar:SISI2014, Kowar:sp2012, Nachman1990}.
The time reversibility requires an approximation of the wave at a ``final stage" in the whole medium. 
This ``internal data" might be known assuming that the reflection of the wave on the measured surface is negligible. 
However, there are many circumstances where this assumption is no longer true.
For example, when the biological tissue under consideration is located inside glass, the waves reflects as in a resonant cavity \cite{Cox:ip2007, Cox:JASA2008, Kunyansky:ip2013}. In this case, measuring or approximating final stage of the wave inside the tissue is impossible. 
We draw the reader's attention to \cite{Linh:SIAM2016} for a method to solve PAT/TAT in a bounded cavity with wave reflection. 
Our contribution in this paper is to apply the quasi-reversibility method to solve the inverse source problem of PAT/TAT for damping and nonhomogeneous media. 
In this case, the model is a full hyperbolic equation in a bounded domain.
By this, the reflection of the waves at the measurement sites is allowed.

The uniqueness and the stability for the inverse source problem for general hyperbolic equations in the damping and inhomogeneous media can be proved by using a Carleman estimate. 
These important results are the extensions of the uniqueness and stability for the inverse problem for a simple hyperbolic equation in \cite{ClasonKlibanov:sjsc2007} in the non-damping case. 
However, since the arguments are very similar to the ones in that paper using Carleman estimate for hyperbolic operators,
for the brevity, we do not write the proof here.

Instead of using the direct optimal control method, to find the initial value of general hyperbolic equations, we derive a system of elliptic partial differential equations, which are considered as an approximation model for our method. 
Solution of this system is the vector consisting of Fourier coefficients of the wave with respect to a special basis.
This system and the given Cauchy boundary data can be solved by the quasi-reversibility method. The quasi-reversibility method was first proposed by Latt\`es and Lions \cite{LattesLions:e1969}
in 1969. Since then it has been studied intensively \cite%
{Becacheelal:AIMS2015, Bourgeois:ip2005, Bourgeois:ip2006,
BourgeoisDarde:ip2010, ClasonKlibanov:sjsc2007, Dadre:ipi2016,
KlibanovKuzhuget:aa2008, Klibanov:jiipp2013}. The application of Carleman
estimates for proofs of convergence of those minimizers was first proposed
in \cite{KlibanovSantosa:SIAMJAM1991} for Laplace's equation. In particular, \cite{KlibanovMalinsky:ip1991} is the first publication where it was proposed to use Carleman estimates to obtain Lipschitz stability of solutions of hyperbolic equations with lateral Cauchy data.
We draw the reader's attention to
the paper \cite{Klibanov:anm2015} that represents a survey of the quasi-reversibility method.
Using a Carleman estimate, we prove Lipschitz-like convergence rate of regularized solutions generated by the quasi-reversibility method to the true solution of that
Cauchy problem. 

 It seems, in theory, that our method of approximation works for any orthonormal basis of $L^2$. 
 However, this observation is not true in the numerical sense. 
That means, the special basis we use to establish the approximation model is crucial.
The basis we use in this paper was first introduced by Klibanov in \cite{Klibanov:jiip2017}, called $\{\Psi_n\}_{n \geq 1}$.
It has a very important property that $\Psi'_n$ is not identically $0$ in an open interval while other bases; for e.g., trigonometric functions and orthonormal polynomials, do not.
In this paper, we prove numerically that choosing this basis is optimal for our method.

As mentioned, we establish in this paper the Lipschitz convergence of the quasi-reversibility method.
Our main contribution to this field is to relax a technical condition on the noise.
In our previous works \cite{NguyenLiKlibanov:IPI2019, LiNguyen:IPSE2019, KlibanovNguyen:ip2019, KhoaKlibanovLoc:SIAMImaging2020, KlibanovLeNguyen:SIAM2020} and references therein, we assumed that the noise contained in the boundary data can be ``smoothly extended" as a function defined on the domain $\Omega$. 
This condition implies that the noise must be smooth.
Motivated by the fact that this assumption is not always true, we employ a Carleman estimate involving the boundary integrals to obtain the new convergence without imposing this ``extension" condition of the noise function.

The paper is organized as follows.
In Section \ref{sec med}, we state the inverse problems under consideration  and derive an approximation model whose solution directly yields their solutions. 
In Section \ref{sec Car}, we introduce some auxiliary results and prove the Carleman estimate, which plays an important role in our analysis.
In Section \ref{sec quasi}, we implement the quasi-reversibility method to solve the system of elliptic equations and prove the convergence of the solution as the noise level tends to $0$. 
Section \ref{sec num} is for the numerical studies. 
Section \ref{sec Klibanov} is to provide some numerical results for $1D-$ problem and to show the significance of the used orthonormal basis.
Finally, section \ref{sec con} is for the concluding remarks.

\section{The Problem statements and a numerical approach} \label{sec med}

Let $\Omega$ be a smooth and bounded domain in $\R^d$, where $d \geq 1$ is the spatial dimension, and $T$ be a positive number.
Let $1 \leq \ba \in C^1(\Omega)$, and $a, \bc \in L^{\infty}(\Omega)$ be functions defined on $\Omega$.
Let $\bb$ be a $d$-dimensional vector valued function in $L^{\infty}(\Omega, \R^d).$
Define the elliptic operator 
\begin{equation}
	L\phi := \Delta \phi + \bb \cdot \nabla \phi + \bc \phi
	\label{L def}
\end{equation}
for all $\phi \in C^2(\overline \Omega)$. Let $p \in H^2_0(\Omega)$ represent a source, we consider the problems of solving $u(\x, t) \in H^2(\Omega \times (0,T))$ generated by the source $p(\x) \in C_0^2(\Omega)$ and subjected to either homogeneous Dirichlet or Neumann boundary condition, which are given by
\begin{equation}
	\left\{
		\begin{array}{rcll}
			\ba(\x)u_{tt}(\x, t) + a(\x) u_t(\x)&=& Lu(\x, t) & (\x, t) \in \Omega \times (0,T)\\
			u_t(\x, 0) &=& 0 &\x \in \Omega,\\
			u(\x, 0) &=& p(\x) &\x \in \Omega,\\
			u(\x, t) &=& 0 &(\x, t) \in \partial \Omega \times [0, T],
		\end{array}
	\right.
	\label{1.1}
\end{equation}
or
\begin{equation}
	\left\{
		\begin{array}{rcll}
			\ba(\x)u_{tt}(\x, t) + a(\x) u_t(\x)&=& Lu(\x, t) &(\x, t) \in \Omega \times (0,T)\\
			u_t(\x, 0) &=& 0 &\x \in \Omega,\\
			u(\x, 0) &=& p(\x) &\x \in \Omega,\\
			\partial_{\nu} u(\x, t) &=& 0 &(\x, t) \in \partial \Omega \times [0, T].
		\end{array}
	\right.
	\label{1,2}
\end{equation} respectively. Our interest is to determine the source $p(\x)$, $\x \in \Omega$, from some boundary observation of the wave. More precisely, the inverse source problems are formulated as:
\begin{problem}
Determine the function $p$ from the measurement of 
	\begin{equation}
		f_1(\x, t) = \partial_{\nu} u(\x, t) \quad (\x, t) \in \partial \Omega \times [0, T]
		\label{1.2}
	\end{equation}
	where $u$ is the solution of \eqref{1.1}.
	\label{pro isp Dir}
\end{problem}

\begin{problem}
Determine the function $p$ from the measurement of 
	\begin{equation}
		f_2(\x, t) =  u(\x, t) \quad (\x, t) \in \partial \Omega \times [0, T]
		\label{1,3}
	\end{equation}
	where $u$ is the solution of \eqref{1,2}.
	\label{pro isp Neu}
\end{problem}

The unique solvability of problems \eqref{1.1} and \eqref{1,2} can be obtained by Garlerkin approximations and energy estimates as in Chapter 7, Section 7.2 in \cite{Evans:PDEs2010}. We now focus on our approach for solving these two inverse problems.
\medskip


Let $\{\Psi_n\}_{n \geq 1}$ be an orthonormal basis of $L^2(0, T).$
The function $u(\x,t)$ can be represented as:
\begin{equation}
	u(\x, t) = \sum_{n = 1}^{\infty} u_n(\x) \Psi_n(t), \quad  \mbox{for all } (\x, t) \in \Omega  \times [0, T]
	\label{2.1}
\end{equation}
where 
\begin{equation}
	u_n(\x) = \int_0^T u(\x, t)\Psi_n(t)dt, \quad n \geq 1.
	\label{2,2}
\end{equation}
 Consider 
\begin{equation}
	\uap(\x, t) := \sum_{n = 1}^N u_n(\x) \Psi_n(t) \quad \mbox{for all } (\x, t) \in \Omega \times [0,T],
	\label{2.2222}
\end{equation}
for some cut-off number $N$.
This number $N$ is chosen numerically such that $\uap$ well-approximates the function $u$, see Section \ref{sec 5.1} for more details.
We have,
\begin{equation}
	\uapt(\x, t) = \sum_{n = 1}^N u_n(\x) \Psi_n'(t),\quad \text{and} \quad \uaptt(\x, t) = \sum_{n = 1}^N u_n(\x) \Psi_n''(t), \quad \mbox{for all } (\x, t) \in \Omega \times [0,T].
	\label{2.2}
\end{equation}
Plugging \eqref{2.2222} and \eqref{2.2} into the first equation of problem  \eqref{1.1}, we get
\begin{equation}
	\sum_{n = 1}^N \ba(\x) u_n(\x) \Psi_n''(t) 
	+ a(\x) \sum_{n = 1}^N u_n(\x) \Psi_n'(t) = \sum_{n = 1}^N Lu_n(\x) \Psi_n(t)
	\label{2.3}
\end{equation}
for all $(\x, t) \in \Omega \times [0,T]$. 
\begin{remark}
	Equation \eqref{2.3} is actually an approximation model. We only solve Problem \ref{pro isp Dir} and Problem \ref{pro isp Neu} in this approximation context. 
	Studying the behavior of \eqref{2.3} as $N \to \infty$ is extremely challenging and out of the scope of the paper.
	In case of interesting, the reader could follow the techniques in \cite{KlibanovLiem:arxiv2020} to investigate the accuracy of   \eqref{2.5} as $N$ tends to $\infty.$
	\label{rem 2.1}
\end{remark}

Since $\{\Psi_n\}_{n \geq 1}$ is an orthonormal basis of $L^2$, multiplying $\Psi_m(t)$ to both sides of \eqref{2.3} and then integrating the resulting equation with respect to $t$, for each $m \in \{1, \dots, N\},$ yields
\begin{equation}
	\sum_{n = 1}^N s_{mn}u_n(\x)  = Lu_m(\x), \quad \mbox{for all } \x \in \Omega
	\label{2.5}
\end{equation}
where
\[
	s_{mn}(\x) = \int_{0}^T \big[\ba(\x)\Psi''_n(t) + a(\x) \Psi_n'(t)\big]\Psi_m(t)dt.
\] Furthermore, from \eqref{2,2} we have 
\begin{equation*}
	\partial_{\nu} u_n(\x) = \int_0^T \partial_{\nu} u(\x, t) \Psi_n(t) dt, \quad \forall \, \x \in \partial \Omega. 
\end{equation*} Therefore, the Cauchy data of $u_n$, for all $n = 1, \ldots, N$ on the boundary $\partial \Omega$ are determined by: 
\begin{enumerate}
	\item Regarding Problem \ref{pro isp Dir}
	\begin{equation}
		\left\{
			\begin{array}{rcl}
				\partial_{\nu} u_n(\x) &=&	\ds \int_0^T f_1(\x, t) \Psi_n(t) dt,\\
				u_n(\x) &=& 0 
			\end{array}
		\right.
	\label{data Neu}
	\end{equation}
	\item Regarding Problem \ref{pro isp Neu}
	\begin{equation}
		\left\{
			\begin{array}{rcl}
				u_n(\x) &=& \ds \int_0^T f_2(\x, t) \Psi_n(t) dt,\\
				\partial_{\nu} u_n(\x) &=&	0.				
			\end{array}
		\right.
		\label{data Dir}
	\end{equation}
\end{enumerate}
The system of elliptic partial differential equations \eqref{2.5} together with Cauchy data either \eqref{data Neu} or \eqref{data Dir} is our approximation model, see Remark \ref{rem 2.1}. It allows to determine coefficients $u_n$, for all $n = 1, \ldots, N$, and then the approximation $u^N(\x,t)$ of $u(\x, t)$.
The source term will be given by $u^N(\x, 0).$
In summary, the numerical method for solving Problem \ref{pro isp Dir} and Problem \ref{pro isp Dir} is described in Algorithm \ref{alg 1} and Algorithm \ref{alg 2}  below respectively.

\begin{algorithm}[h!]
	\caption{\label{alg 1} A numerical method to solve Problem \ref{pro isp Dir}}
	\begin{algorithmic}[1]
	\State \label{step basis} Choose the basis $\{\Psi\}_{n \geq 1}$ and a cut-off number $N$.
	\State \label{step 1} Solve the system \eqref{2.5} with Cauchy data \eqref{data Neu} for the vector valued function 
	\[\bu^{\rm comp}(\x) = (u_1^{\rm comp}, \dots, u_N^{\rm comp}), \quad \x \in \Omega.\]  
	\State \label{step 4} The function $p^{\rm comp}(\x)$ is given by \[p^{\rm comp}(\x) = u^{\rm comp}(\x, 0) = \sum_{n = 1}^N u_n^{\rm comp}(\x) \Psi_n(0), \quad \x \in \Omega.\]
 	\end{algorithmic}
\end{algorithm}

\begin{algorithm}[h!]
	\caption{\label{alg 2} A numerical method to solve Problem \ref{pro isp Neu}}
	\begin{algorithmic}[1]
	\State Choose the basis $\{\Psi\}_{n \geq 1}$ and a cut-off number $N$.
	\State Solve the system \eqref{2.5} with Cauchy data \eqref{data Dir}  for the vector valued function 
	\[\bu^{\rm comp}(\x) = (u_1^{\rm comp}, \dots, u_N^{\rm comp}), \quad \x \in \Omega.\]  
	\State  The function $p^{\rm comp}(\x)$ is given by \[p^{\rm comp}(\x) = u^{\rm comp}(\x, 0) = \sum_{n = 1}^N u_n^{\rm comp}(\x) \Psi_n(0), \quad \x \in \Omega.\]
 	\end{algorithmic}
\end{algorithm}

\begin{remark}
\label{rm2.1}
	In Step \ref{step basis} of Algorithm \ref{alg 1} and Algorithm \ref{alg 2}, we choose the basis $\{\Psi_n\}_{n \geq 1}$ taken from  \cite{Klibanov:jiip2017}. 
	The cut-off number $N$ is chosen numerically. 
	More details will be discussed in Section \ref{sec num}. 
	In Step \ref{step 1} of these algorithms, we apply the quasi-reversibility method to solve \eqref{2.5} and \eqref{data Neu} and \eqref{2.5} and \eqref{data Dir}.
	The analysis about about the quasi-reversibility method and its convergence as the noise in the given data tends to $0$ are discussed in Section \ref{sec quasi}.
\end{remark}

As mentioned in Remark \ref{rm2.1} that solving \eqref{2.5} and \eqref{data Neu} for Problem \ref{pro isp Dir} and solving \eqref{2.5} and \eqref{data Dir} for Problem \ref{pro isp Neu} are interesting when the given data in \eqref{1.2} and \eqref{1,3} contain noise. 
We employ the quasi-reversibility method to do so.
Let $\epsilon$ be a small positive number playing the role of the regularization parameter.
To solve \eqref{2.5} and \eqref{data Neu} we minimize the following mismatch functional
\begin{multline}
	J_1(u_1, \dots, u_N) = \sum_{m = 1}^N\int_{\Omega}  \big|Lu_m - \sum_{n = 1}^Ns_{mn} u_n\big|^2d\x 
	\\
	+\sum_{m = 1}^N \int_{\partial \Omega}  \big|\partial_\nu u_m - \int_0^T f_1(\x, t) \Psi_n(t) dt\big|^2 d\sigma 
	+ \epsilon \|\bu\|_{H^2(\Omega)^N}^2
	\label{2.9}
\end{multline}
where $\bu = (u_1, \dots, u_N)$ is in $H^2(\Omega)^N$ satisfying $\bu(\x) = 0$ for all $\x \in \partial \Omega$.
To solve \eqref{2.5} and \eqref{data Dir}, we minimize the following mismatch functional
\begin{multline}
	J_2(u_1, \dots, u_N) = \sum_{m = 1}^N\int_{\Omega} \big|Lu_m - \sum_{n = 1}^Ns_{mn} u_n\big|^2d\x 
	\\
	+\sum_{m = 1}^N \int_{\partial \Omega} \big| u_m - \int_0^T f_2(\x, t) \Psi_n(t) dt\big|^2 d\sigma 
	+ \epsilon \|\bu\|_{H^2(\Omega)^N}^2
	\label{2.99}
\end{multline}
where $\bu = (u_1, \dots, u_N)$ is in $H^2(\Omega)^N$ satisfying $\partial_\nu\bu(\x) = 0$ for all $\x \in \partial \Omega$.

In the following sections, we prove that $J_1$ has a unique minimizer and that minimizer converges to the true solution of \eqref{2.5} and \eqref{data Neu} as the noise level tends to $0$.
The results for $J_2$ can be obtained in the same manner. 
We introduce some auxiliary results in the next section.

\section{A Carleman estimate} \label{sec Car}


Let $X$ be a number in $(0, 1)$. We denote 
\[
	\widetilde \Omega := \Big\{\tilde \x = (\tilde x_1, \ldots, \tilde x_d): 0 < \tilde x_1 + X^{-2} \sum_{i = 2}^d \tilde x_i^2 < 1\Big\}.
\]
Then, there exists $\alpha \in (0, 1/2)$ such that the function 
\begin{equation}
\label{psi}
	\psi(\tilde \x) := \tilde x_1 + \frac{1}{2X^2} \sum_{i = 2}^d \tilde x_i^2 + \alpha < 1, \; \forall \, \tilde \x \in \widetilde{\Omega}.	
\end{equation}
We have the following lemma.

\begin{Lemma}
	There are two positive constants $\lambda_0$ and $\beta_0$ depending only on $\alpha$ such that for all $\lambda > \lambda_0$ and $\beta > \beta_0$, 
	we have
	\begin{equation}
		\frac{\lambda \beta}{X^2} e^{2\lambda \psi^{-\beta}}|\nabla \phi|^2 
		+ \lambda^3 \beta^4 \psi^{-2\beta - 2} e^{2\lambda \psi^{-\beta}}|\phi|^2
		\leq 
		-\frac{C\lambda \beta}{X^2} e^{2\lambda \psi^{-\beta}} \phi |\Delta \phi|^2
		+ C \psi^{\beta + 2} e^{2\lambda \psi^{-\beta}}|\Delta \phi|^2 + \Div \Phi
		\label{Lav}
	\end{equation}
	for all function $\phi \in C^2(\overline{\widetilde{\Omega}})$
	where the vector $\Phi$ satisfies
	\begin{equation}
		|\Phi| \leq C e^{2\lambda \psi^{-\beta}} 
		\Big(
			\frac{\lambda \beta}{X} |\nabla \phi|^2 + \frac{\lambda^3 \beta^3}{X^3} \psi^{-2\beta - 2}|\phi|^2
		\Big).
		\label{Phi}
	\end{equation}
\label{lemma Lavrentev}
\end{Lemma}
Lemma \ref{lemma Lavrentev} is a direct consequence of \cite[Chapter 4, \S 1, Lemma 3]{Lavrentiev:AMS1986} in which the function $\phi$ is independent of the time variable. Let $\tilde R$ be a positive number such that $\Omega \subset B(0, \tilde R)$, where $B(0,\tilde R)$ denotes a ball of the center at $0$ and the radius $\tilde R$. Let $p$ and $q$ be two positive numbers such that 
\begin{equation}
\label{TransDomain}
	p \, B(0,\tilde R) + q\, \e_1 \subset \widetilde \Omega,
\end{equation} where $\e_1$ is the unit direction vector of $x_1$ axis. For any $\x \in \Omega$, we define $\tilde \x: = p \x + q \e_1$, then \eqref{TransDomain} yields $\tilde \x \in \widetilde \Omega$. By modifying constant $C$ in Lemma \ref{lemma Lavrentev} (using $C p^2$ instead of $C$) we have that the Lemma \ref{lemma Lavrentev} holds true in domain $\Omega$. From now on, we apply Lemma \ref{lemma Lavrentev} for all function in the space $H^2({\Omega})$. The following result plays an important role in our analysis.

\begin{Proposition}[Carleman estimate]
	There exist $\lambda_0 > 1$, $\beta_0 > 1$, both of which only depend on $\ba$ and $\alpha$ such that for all $\lambda > \lambda_0$ and $\beta > \beta_0$ 
	for all $\phi \in H^2(\Omega)$, we have
	\begin{multline}
			\int_{\Omega} \psi^{\beta + 2}e^{2\lambda \psi^{-\beta}} |\Delta \phi|^2 
			 \geq \int_{\Omega} C e^{2\lambda \psi^{-\beta}} \big[\lambda\beta  |\nabla \phi|^2 + \lambda^3 \beta^4 \psi^{-2\beta - 2} |\phi|^2 \big]d\x
			 \\
			-  C\int_{\partial \Omega}e^{2\lambda \psi^{-\beta}}
			\big[
				\lambda \beta |\nabla \phi|^2 + \lambda^3 \beta^3\psi^{-2\beta - 2}|\phi|^2
			\big]d\sigma
			\label{Car est}
		\end{multline}
		where $C$ is a generic constant depending only on $\ba,$ $d$, $\Omega$, $X$ and $\alpha$.
		\label{prop Car}
\end{Proposition}

\begin{proof}
	Let $\lambda_0$ and $\beta_0$ be as in Lemma \ref{lemma Lavrentev}.
	Fix $\lambda > \lambda_0$ and $\beta > \beta_0.$
	Using the inequality $ab \leq \lambda \beta \psi^{-\beta - 2} a^2 + \psi^{\beta + 2}b^2/(2\lambda \beta)$, for all $\phi \in C^2(\overline \Omega),$ we have
	\begin{equation}
		-\lambda \beta e^{2\lambda \psi^{-\beta}} \phi \Delta \phi
		\leq \lambda^2 \beta^2 \psi^{-\beta - 2} e^{2\lambda \psi^{-\beta}} |\phi|^2 + \frac{1}{2} \psi^{\beta + 2}e^{2\lambda \psi^{-\beta}} |\Delta \phi|^2
		\quad \mbox{in } \Omega.
		\label{3-3}
	\end{equation}
	Combining \eqref{Lav} and \eqref{3-3}, since $\psi < 1$, we have
	\begin{equation}
		\psi^{\beta + 2}e^{2\lambda \psi^{-\beta}} |\Delta \phi|^2  		
		\geq  
		C e^{2\lambda \psi^{-\beta}} \Big[\lambda\beta  |\nabla \phi|^2 + \lambda^3 \beta^4 \psi^{-2\beta - 2}|\phi|^2  -\Div(\Phi)\Big]
		\quad \mbox{in } \Omega
		\label{3-5}
	\end{equation}
	where $\Phi$ is a vector satisfying \eqref{Phi}.		Integrate \eqref{3-5} over $\Omega$ and apply the divergence theorem. Recaling \eqref{Phi}, we obtain \eqref{Car est}.		
\end{proof}

The Carleman estimate \eqref{Car est} plays a key role for us to estimate the error of the solution to the inverse problem assuming that the given data contains noise.

\section{The quasi-reversibility method} \label{sec quasi}

As mentioned in section \ref{sec med}, we only prove the convergence of the quasi-reversibility method to solve \eqref{2.5} with Cauchy boundary data \eqref{data Neu}.
In this case, the objective functional $J_1$, now named as $J$, see \eqref{2.9}, is written as
 \begin{multline}
	J(u_1, \dots, u_N) = \sum_{m = 1}^N\int_{\Omega} |Lu_m - \sum_{n = 1}^Ns_{mn} u_n|^2d\x 
	\\
	+\sum_{m = 1}^N \int_{\partial \Omega}  \Big| \partial_{\nu} u_m - \int_0^T f_1(\x, t) \Psi_m(t) dt\Big|^2 d\sigma 
	+ \epsilon \sum_{m = 1}^N \|u_m\|^2_{H^2(\Omega)}
	\label{3.2}
\end{multline}
subject to the constraint $u_1 = \ldots = u_N  = 0$ on $\partial \Omega$. Let
\[
	\H = \{\f \in H^2(\Omega)^N: \f|_{\partial \Omega} = 0\}
\]
be a closed  subspace of $H^2(\Omega)^N$.
It is clear that $J$ is strictly convex in $\H$. 
We now prove that $J$ is has a unique minimizer in $\H$.

\begin{Proposition}
	The functional $J$ has a unique minimizer in $\H$.
	\label{prop minimizer}
\end{Proposition}
\begin{proof}
	Let 
	\[
		e = \inf \big\{J(\bu): \bu = (u_1, \dots, u_n) \in \H\big\} \, \geq 0
	\]
	and $\{\bu_{\frak{i}}\}_{\frak{i} \geq 1}$ be a sequence satisfying
	\[
		\lim_{\frak{i} \to \infty} J(\bu_\frak{i}) = e.
	\]
	Then $\{\bu_\frak{i}\}_{\frak{i} \geq 1}$ is bounded in $\H$.
	In fact, by contradiction, assume that $\{\bu_\frak{i}\}_{\frak{i} \geq 1}$ is unbounded. Then, there exists a subsequence, still named as $\{\bu_\frak{i}\}_{\frak{i} \geq 1}$, satisfying $\lim_{\frak{i} \to \infty}\|\bu_\frak{i}\|_{H^2(\Omega)^N} = \infty.$ Hence, 
	\[
		e = \lim_{\frak i \to \infty}J(u_\frak i) \geq \lim_{\frak i \to \infty} \epsilon  \|\frak u_{\frak{i}}\|^2_{H^2(\Omega)^N} = \infty,
	\]
	which is impossible.	Due to the boundedness of $\{\bu_\frak{i}\}_{\frak{i} \geq 1}$ in $\H$, there exists a subsequence of  $\{\bu_\frak{i}\}_{\frak{i} \geq 1}$, still named as $\{\bu_\frak{i}\}_{\frak{i} \geq 1}$, which weakly converges to a function $\bu_0$ in $\H$. That implies $ \{\partial_\nu \bu_{\frak{i}}\}_{\frak{i} \geq 1}$ weakly converges to $\partial_{\nu} \bu_0$ in $H^{1/2}(\partial \Omega)^N$, and therefore, $\{\partial_\nu \bu_{\frak{i}}\}_{\frak{i} \geq 1}$ strongly converges to $\partial_{\nu} \bu_0$ in $L^2(\partial \Omega)$ by the compact imbedding of $H^{1/2}(\partial \Omega)$ into $L^2(\partial \Omega)$. Furthermore, the fact that $\{\bu_\frak{i}\}_{\frak{i} \geq 1}$ weakly converges  $\bu_0$ in $\H$ implies  $\Big(L(\bu_\frak{i})_m - \sum_{n = 1}^N s_{mn}(\bu_i)_n\Big)_{m = 1}^N$ weakly converges to $\Big(L(\bu_0)_m - \sum_{n = 1}^N s_{mn} (\bu_0)_n\Big)_{m = 1}^N$ in $L^2(\Omega)^N$.
	As a result,
	\[
		\sum_{m = 1}^N\int_{\Omega} \Big|L(\bu_0)_m - \sum_{n = 1}^N s_{mn} (\bu_0)_n\Big|^2d\x 
		\leq \limsup_{i \to \infty} \sum_{m = 1}^N\int_{\Omega}\Big|L(\bu_i)_m - \sum_{n = 1}^N s_{mn} (\bu_i)_n\Big|^2d\x.
	\]
	Therefore
	\[
		J(\bu_0) \leq \limsup_{\frak{i} \to \infty} J(\bu_\frak{i}) = e.
	\]
	Thus $\bu_0$ is a minimizer of $J.$
	The uniqueness of $\bu_0$ is deduced from the strict convexity of $J$.
\end{proof}

\begin{Definition}
	The unique minimizer, denoted by \, $\bu^{\rm min} = (u_1^{\min}, \dots, u_N^{\min})$, of functional $J$ is said to be the regularized solution of the problem \eqref{2.5} -- \eqref{data Neu}.
\end{Definition}

We now assume that the measured data contain noise, with a noise level $\delta$. We next study the convergence of $\bu^{\min}$ as noise level $\delta$ tends to 0. Let denote $f_1^{\delta}(\x, t)$ the noisy data and $f_1^*(\x,t)$ the corresponding noiseless data, $(\x, t) \in \partial\Omega \times [0,T]$.
By noise level $\delta$, we mean
 \[
 	\left(\int_0^T\int_{\partial \Omega} |f_1^{\delta}(\x, t) - f_1^*(\x, t)|^2d\sigma dt \right)^{1/2}< \delta.
 \]
 Since the truncation number $N$ is a finite number, we can write
\begin{equation}
 	\sum_{m = 1}^N\int_{\partial \Omega}\Big|\int_0^T f_1^{\delta}(\x, t) \Psi_m(t) dt - \int_0^T f_1^*(\x, t) \Psi_m(t) dt\Big|^2 \leq C \delta^{2},
	\label{error}
 \end{equation}
 where $C$ is a generic constant depending only on $N, \ba, \bb$ and $c$. For each $m \in \{1, \dots, N\}$, define
\begin{equation}
	\frak{f}^*_m(\x) = \int_0^T f_1^*(\x, t) \Psi_m(t) dt \quad \mbox{and} \quad \frak{f}^\delta_m(\x) = \int_0^T f_1^{\delta}(\x, t) \Psi_m(t) dt,
	\label{frak}
\end{equation}
for all $\x \in \partial \Omega.$ The following theorem guarantees the Lipschitz stability of the reconstructed method with respect to noise.
 \begin{Theorem}
 	Let $\bu_{\min}^\delta \in \H$ be the minimizer of the functional 
	 \begin{equation}
	J_{\delta}(u_1, \dots, u_N) = \sum_{m = 1}^N\Big[\int_{\Omega} |Lu_m - \sum_{n = 1}^Ns_{mn} u_n|^2d\x 
	+\int_{\partial \Omega}  |\partial_\nu u_{m} - \frak{f}^{\delta}_m|^2d\sigma\Big]
	+ \epsilon \sum_{m = 1}^N\|u_m\|^2_{H^2(\Omega)}.
	\label{4.1}
\end{equation}
Assume that the system
\begin{equation}
	\left\{
		\begin{array}{rcll}
			L u_m -\ds \sum_{n = 1}^N s_{mn}u_n &=& 0 &\mbox{in } \Omega,\\
			\partial_\nu u_m  &=& \frak{f}^*_m &\mbox{on } \partial \Omega,\\
			u_m &=& 0 &\mbox{on } \partial \Omega
		\end{array}
	\right. 
	\quad m \in \{1, \dots, N\}
	\label{4.2}
\end{equation}
has the unique solution $\bu^* = (u_1^*, \dots, u_N^*) \in \H.$
Then, 
\begin{equation}
	\|\bu_{\min}^\delta - \bu^*\|_{H^1(\Omega)^N}^2 \leq C \big(\delta^2 + \epsilon \|\bu^*\|_{H^2(\Omega)^N}^2\big)
	\label{4.6}
\end{equation}
where $C$ is a generic constant depending only on $N, \ba, \bb$ and $c$.
As a result, let $u^{\rm comp}$ and $u^*$ be the functions obtained by \eqref{2.2222} with $(u_1, \dots, u_N)$ replaced by $\bu^{\delta}_{\rm min}$ and $\bu^*$ respectively and
let $p^{\rm comp}(\x)$ and $p^{*}(\x)$ be $u^{\rm comp}(\x, 0)$ and $u^*(\x, 0)$ respectively, $\x \in \Omega$. 
We have
\begin{equation}
	\|p^{\rm comp} - p^*\|_{H^1(\Omega)}^2 \leq C \big(\delta^2 + \epsilon \|\bu^*\|_{H^2(\Omega)^N}^2\big).
	\label{4.7}
\end{equation}
\label{thm 4.1}
 \end{Theorem}
\begin{proof}
	Due to \eqref{2.2222}, we have
	\[
		p^{\rm comp}(\x) = \sum_{n  = 1}^N u_n^{\rm comp}(\x)\Psi_n(0) 
		\quad \mbox{and} \quad
		p^*(\x) = \sum_{n  = 1}^N u_n^*(\x)\Psi_n(0).
	\] $\x \in \Omega$.
	Hence, \eqref{4.6} implies \eqref{4.7}.
	It is sufficient to prove \eqref{4.6}.
	Since $\bu_{\min}^{\delta} = (u_1, \dots, u_N)$ is the minimizer of $J_\delta$, for all $\bh = (h_1, \dots, h_N) \in \H,$
	\begin{multline}
		\sum_{m = 1}^N\Big\langle L u_m - \sum_{n = 1}^Ns_{mn} u_n,  L h_m - \sum_{n = 1}^N s_{mn} h_n\Big\rangle_{L^2(\Omega)}
		\\
		+ \sum_{m = 1}^N \Big\langle \partial_{\nu }u_m - \frak{f}^\delta_m,  \partial_{\nu }h_m\Big\rangle_{L^2(\partial \Omega)}
		+ \epsilon \sum_{m = 1}^N \langle u_m, h_m\rangle_{H^2(\Omega)}
		= 0.
		\label{4.4}
	\end{multline}
	Since $\bu^* = (u_1^*, \dots, u_N^*)$ is the true solution to \eqref{4.2}, for all $\bh = (h_1, \dots, h_N) \in \H,$
	\begin{multline}
		\sum_{m = 1}^N\Big\langle L u_m^* - \sum_{n = 1}^Ns_{mn} u_n^*,  L h_m - \sum_{n = 1}^N s_{mn} h_n\Big\rangle_{L^2(\Omega)}
		\\
		+ \sum_{m = 1}^N \Big\langle \partial_\nu u_m^* - \frak{f}^*_m,  \partial_\nu h_m\Big\rangle_{L^2(\partial \Omega)} 
		+ \epsilon \sum_{m = 1}^N \langle u_m^*, h_m\rangle_{H^2(\Omega)}
		= \epsilon \sum_{m = 1}^N \langle u_m^*, h_m\rangle_{H^2(\Omega)}.
		\label{4.5}
	\end{multline}
	Hence, by subtracting \eqref{4.4} from \eqref{4.5}  and setting $\bh = (h_1, \dots, h_N) = \bu^\delta_{\min} - u^*$, we have
	\begin{multline*}
		\sum_{m = 1}^N\Big \|L h_m - \sum_{n = 1}^N s_{mn} h_n\Big\|_{L^2(\Omega)}^2
		+ \sum_{m = 1}^N \langle \partial_\nu h_m - (\frak{f}^\delta_m - \frak{f}^*_m), \partial_\nu h_m\rangle_{L^2(\partial \Omega)}
		\\
		+ \epsilon \sum_{m = 1}^N \|h_m\|_{H^2(\Omega)}^2
		= -\epsilon \sum_{m = 1}^N \langle u_m^*, h_m\rangle_{H^2(\Omega)}.
	\end{multline*}	
	Equivalently,
	\begin{multline*}
		\sum_{m = 1}^N\Big \|\Delta h_m + \bb \cdot \nabla h_m + c h_m - \sum_{n = 1}^N s_{mn} h_n\Big\|_{L^2(\Omega)}^2
		\\
		+ \sum_{m = 1}^N \|\partial h_m\|_{L^2(\partial \Omega)}^2
		+ \epsilon \sum_{m = 1}^N \|h_m\|_{H^2(\Omega)}^2		
		= \sum_{m = 1}^N \langle  \frak{f}^\delta_m - \frak{f}^*_m, \partial_\nu h_m\rangle_{L^2(\partial \Omega)} - \epsilon \sum_{m = 1}^N \langle u_m^*, h_m\rangle_{H^2(\Omega)}. 
	\end{multline*}	
	Using the inequality $|ab| \leq \frac{a^2}{2} + \frac{b^2}{2}$, we have
	\begin{multline}
		\sum_{m = 1}^N\Big \|\Delta h_m + \bb \cdot \nabla h_m + c h_m - \sum_{n = 1}^N s_{mn} h_n\Big\|_{L^2(\Omega)}^2
		+\frac{1}{2} \sum_{m = 1}^N \|\partial_\nu h_m\|_{L^2(\partial \Omega)}^2
		\\
		+ \frac{\epsilon}{2} \sum_{m = 1}^N \|h_m\|_{H^2(\Omega)}^2
		\leq
		\frac{1}{2} \sum_{m = 1}^N \|\frak{f}^\delta_m - \frak{f}^*_m\|^2_{L^2(\partial \Omega)} 
		+ \frac{\epsilon}{2}\sum_{m = 1}^N \|u_m^*\|_{H^2(\Omega)}^2.
		\label{4.8}
	\end{multline}	
	It follows from \eqref{error}, \eqref{frak} and \eqref{4.8} that
	\begin{equation}
		\sum_{m = 1}^N\int_{\Omega}\Big|\Delta h_m + \bb \cdot \nabla h_m + c h_m - \sum_{n = 1}^N s_{mn} h_n\Big|^2 d\x
		+ \sum_{m = 1}^N \| \partial_{\nu }h_m\|_{L^2(\partial \Omega)}^2
		\leq C \big(\delta^2 + \epsilon \|\bu^*\|_{H^2(\Omega)^N}^2\big)
		\label{4.10}
	\end{equation}
	for a constant $C > 0.$
	It follows from \eqref{4.10} that
	\begin{equation}
		 \sum_{m = 1}^N \|\partial_{\nu} h_m\|_{L^2(\partial \Omega)}^2
		\leq C \big(\delta^2 + \epsilon \|\bu^*\|_{H^2(\Omega)^N}^2\big).
		\label{4.1111}
	\end{equation}
	Since $h_m = 0$ on $\partial \Omega$, $1 \leq m \leq N$, the tangent derivative of $h_m$ on $\partial \Omega$ is $0$. 
	Hence, by \eqref{4.1111}
	\begin{equation}
		 \sum_{m = 1}^N \|\nabla h_m\|_{L^2(\partial \Omega)}^2
		\leq C \big(\delta^2 + \epsilon \|\bu^*\|_{H^2(\Omega)^N}^2\big).
		\label{4.11}
	\end{equation}
	Recall $\lambda_0$, $\beta_0$ as in Lemma \ref{lemma Lavrentev} and the function $\psi$ as in \eqref{psi}.
	Fix $\beta  = \beta_0$.
	Applying the inequality $(a - b)^2 \geq a^2/2 - b^2$, we have for all $\lambda > \lambda_0$
	\begin{align*}
		\sum_{m = 1}^N&\int_{\Omega}\Big|\Delta h_m + \bb \cdot \nabla h_m + c h_m - \sum_{n = 1}^N s_{mn} h_n\Big|^2 d\x
		\\
		&\geq \min_{\x \in \overline \Omega}\big\{e^{-2\lambda \psi^{-\beta}}\psi^{-\beta - 2}\big\}
		\sum_{m = 1}^N\int_{\Omega} e^{2\lambda \psi^{-\beta}}\psi^{\beta + 2}\Big|\Delta h_m + \bb \cdot \nabla h_m + c h_m - \sum_{n = 1}^N s_{mn} h_n\Big|^2 d\x
		\\
		&\geq  \min_{\x \in \overline \Omega}\big\{e^{-2\lambda \psi^{-\beta}}\psi^{-\beta - 2}\big\}
		\Big[\frac{1}{2}\sum_{m = 1}^N\int_{\Omega} e^{2\lambda \psi^{-\beta}}\psi^{\beta + 2}|\Delta h_m|^2 d\x
		\\
		&\hspace{5cm}- \sum_{m = 1}^N \int_{\Omega} e^{2\lambda \psi^{-\beta}}\psi^{\beta + 2}\Big|\bb \cdot \nabla h_m + c h_m - \sum_{n = 1}^N s_{mn} h_n\Big|^2 d\x\Big].
	\end{align*}	
	Thus, by \eqref{4.10},
	\begin{multline*}
		\min_{\x \in \overline \Omega}\big\{e^{-2\lambda \psi^{-\beta}}\psi^{-\beta - 2}\big\}
		\Big[
		\sum_{m = 1}^N\int_{\Omega} e^{2\lambda \psi^{-\beta}}\psi^{\beta + 2}|\Delta h_m|^2 d\x 
		\\
		- 2\sum_{m = 1}^N \int_{\Omega} e^{2\lambda \psi^{-\beta}}\psi^{\beta + 2}\Big|\bb \cdot \nabla h_m + c h_m - \sum_{n = 1}^N s_{mn} h_n\Big|^2 d\x
		\Big]
		\leq C\big(\delta^2 + \epsilon \|\bu^*\|_{H^2(\Omega)^N}^2\big).
	\end{multline*}
	Applying the Carleman estimate \eqref{Car est}, we have
	\begin{multline*}
		\min_{\x \in \overline \Omega}\big\{e^{-2\lambda \psi^{-\beta}}\psi^{-\beta - 2}\big\}
		\Big[
		C\sum_{m = 1}^N\int_{\Omega} e^{2\lambda \psi^{-\beta}}[\lambda \beta |\nabla h_m|^2 + \lambda^3 \beta^4 \psi^{-2\beta - 2} |h_m|^2] d\x 
		\\
		- C\int_{\partial \Omega}  e^{2\lambda \psi^{-\beta}} \big[
				\lambda \beta |\nabla h_m|^2 + \lambda^3 \beta^3\psi^{-2\beta - 2}|h_m|^2
			\big]d]\sigma
		\\
		- 2\sum_{m = 1}^N \int_{\Omega} e^{2\lambda \psi^{-\beta}}\psi^{\beta + 2}\Big|\bb \cdot \nabla h_m + c h_m - \sum_{n = 1}^N s_{mn} h_n\Big|^2 d\x
		\Big]
		\leq C\big(\delta^2 + \epsilon \|\bu^*\|_{H^2(\Omega)^N}^2\big).
	\end{multline*}
 Since $\beta = \beta_0$ fixed, 	choosing $\lambda$ sufficiently large,we have
	\begin{equation}
\sum_{m = 1}^N \int_{\Omega} \big[\lambda  |\nabla h_m|^2 + \lambda^3|h_m|^2\big]d\x
		\leq 
		C\int_{\partial \Omega} |\lambda |\nabla h_m|^2 + \lambda^3 |h_m|^2d\sigma
		+  C\big(\delta^2 + \epsilon \|\bu^*\|_{H^2(\Omega)^N}^2\big)
		\label{4.14}
	\end{equation}
	Since ${\bf h} = \bu^\delta - \bu^* \in \H$, $h_m = 0$ on $\partial \Omega.$ Hence, we obtain \eqref{4.6} by using \eqref{4.11} and \eqref{4.14}.
\end{proof}

\begin{remark}
	The conclusion of Theorem \ref{thm 4.1} is similar to some theorems about the quasi-reversibility method we have developed, see e.g. \cite[Theorem 5.1]{NguyenLiKlibanov:IPI2019} and \cite[Theorem 4.1]{LiNguyen:IPSE2019}.
	The main difference is that in Theorem \ref{thm 4.1}, we relax a technical condition that there exists an error vector valued function $\mathcal E$, well-defined in the whole $\Omega$, such that  $\partial_\nu \mathcal E = \frak f^\delta - \frak f^*$ and $\|\mathcal E\|_{H^2(\Omega)} = O(\delta).$  
\end{remark}

The analysis for the quasi-reversibility method to solve \eqref{2.5} and \eqref{data Dir} is similar to the arguments above. 
We do not repeat the proof here.

\section{Numerical studies}  \label{sec num}

In this section, we set the dimension $d = 2$ and $\Omega = (-R, R)^2$ with $R = 1$.
Define a grid of points on $\overline \Omega$ as
\[
	\mathcal G = \{(x_i = -R + (i - 1)h_\x, \; y_j = -R + (j - 1)h_\x): 1 \leq i, j \leq N_\x\}
\] where $h_\x = 2R/(N_\x - 1)$ with $N_\x = 81.$
We set $T = 2$. On $[0, T]$, we also define the uniform partition
\[
	\mathcal T = \{t_i = (i - 1)h_t: 1 \leq i \leq N_T\}
\]
 where $h_t = T/(N_T - 1)$. In our computation, $N_T = 201$.
To generate the simulated data, we solve \eqref{1.1} with
\[
	\left\{
		\begin{array}{rcl}
			\ba(x, y) &=& 1 + \sin^2(x^2 + y^2),\\	
			\bb(x, y) &=& (2, 1),\\
			c(x, y) &=& \cos(x^2 + y^2),\\
			e(x, y) &=& 0.5[\cos(x^2 + y^2) + \sin(x^2 + y^2)]
		\end{array}
	\right. \quad (x, y) \in \Omega
\]
by the finite difference method in the implicit scheme. 
Let $u^*(\x, t)$, $\x \in \mathcal G$ and $t \in \mathcal T$, be the obtained numerical solution. 
We can extract the data $\mathcal B u^*(\x, t)$ and $\mathcal Fu^*(\x, t)$ on $(\partial \Omega \times [0, T]) \cap (\mathcal G \times \mathcal T)$. These functions serve as the data without noise.
For $\delta > 0$, the noisy data are given by
\[
	\mathcal B u(\x, t) = \mathcal B u^*(\x, t)(1 + \delta {\rm rand}(\x, t)) \quad \mbox{and } \quad
	\mathcal F u(\x, t) = \mathcal F u^*(\x, t)(1 + \delta {\rm rand}(\x, t))
\]
where ${\rm rand}$ is the function taking  uniformly distributed random numbers in $[-1, 1].$

\subsection{Implementation}
\label{sec 5.1}

We present in details the implementation of Steps \ref{step basis} and \ref{step 1} of Algorithm \ref{alg 1} to solve Problem \ref{pro isp Dir} while the other steps can be implemented directly.
The implementation for Problem  \ref{pro isp Neu} can be done in the same manner.

{\it Step \ref{step basis} in Algorithm \ref{alg 1}.} In our numerical studies, we employ the basis $\{\Psi_n\}_{n \geq 1}$ that was first introduced by Klibanov in \cite{Klibanov:jiip2017}.
For each $n \geq 1,$ we define $\Phi_n(t) := t^{n - 1}e^{t - T/2}.$
The set $\{\Phi_n: n \geq 1\}$ is complete in $L^2(0, T)$.
We apply the Gram-Schmidt orthonormalization process on this set to obtain the orthonormal basis $\{\Psi_n\}_{n \geq 1}$ of $L^2(0, T).$
\begin{remark}
The basis $\{\Psi_n\}_{n \geq 1}$ was successfully used very often in our research group to solve a long list of inverse problems including the nonlinear coefficient inverse problems for elliptic equations \cite{VoKlibanovNguyen:IP2020} and parabolic equations \cite{Nguyen:arxiv2019, KlibanovNguyen:ip2019, KhoaKlibanovLoc:SIAMImaging2020, Khoaelal:IPSE2020}, and ill-posed inverse source problems for elliptic equations \cite{NguyenLiKlibanov:IPI2019} and parabolic equations \cite{LiNguyen:IPSE2019}, transport equations \cite{KlibanovLeNguyen:SIAM2020} and full transfer equations \cite{KlibanovAlexeyNguyen:SISC2019}. 
Another reason for us to employ this basis rather than the well-known basis of the Fourier series is that the first elements of this basis is a constant.
Hence, when we plug \eqref{2.2} into \eqref{1.1}, the information of $u_1(\x) \Psi_1''(t)$ will be lost. As a result, the contribution of $u_1(\x)$ in \eqref{2.5} is less than the that of the corresponding $u_1(\x)$ obtained by the basis $\{\Psi_n\}_{n \geq 1}$.
\label{basis Psi}
\end{remark}

To choose $N$, we numerically compare $u^*(\x, 0)$ and $\ds\sum_{n = 1}^N u^*_n(\x)\Psi_n(0)$ for
$
	\x \in \Omega
$ where $u^*$ is the true solution to \eqref{1.1} and the source $p(\x)$ is given in Example 1 below. The number $N$ is chosen such that the error \[\Big|u^*(\x, 0) - \ds\sum_{n = 1}^N u^*_n(\x)\Psi_n(0)\Big|\] is small enough. 
We perform this procedure and choose $N = 35$, see Figure \ref{fig choose N}. This cut off number is used for all numerical examples in the paper.

\begin{figure}[h!]
\begin{center}
	\subfloat[$N = 15$]{\includegraphics[width=.3\textwidth]{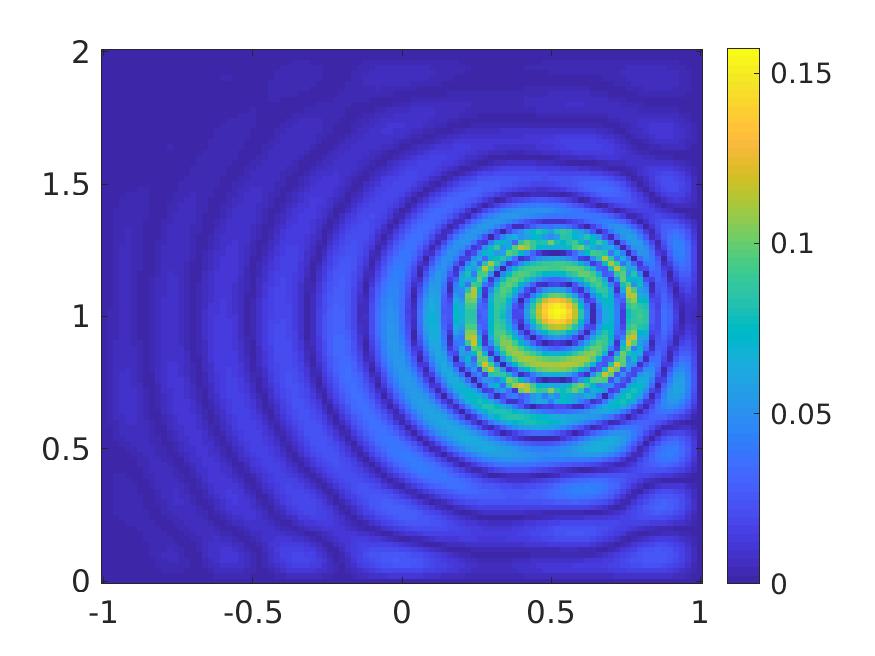}}
	\subfloat[$N = 20$]{\includegraphics[width=.3\textwidth]{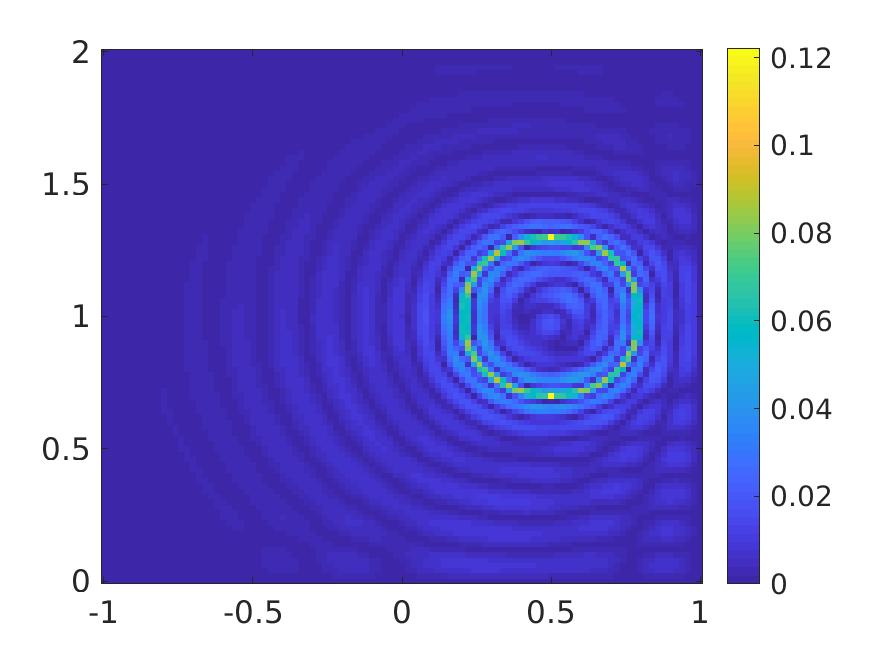}}
	\subfloat[N = 35]{\includegraphics[width=.3\textwidth]{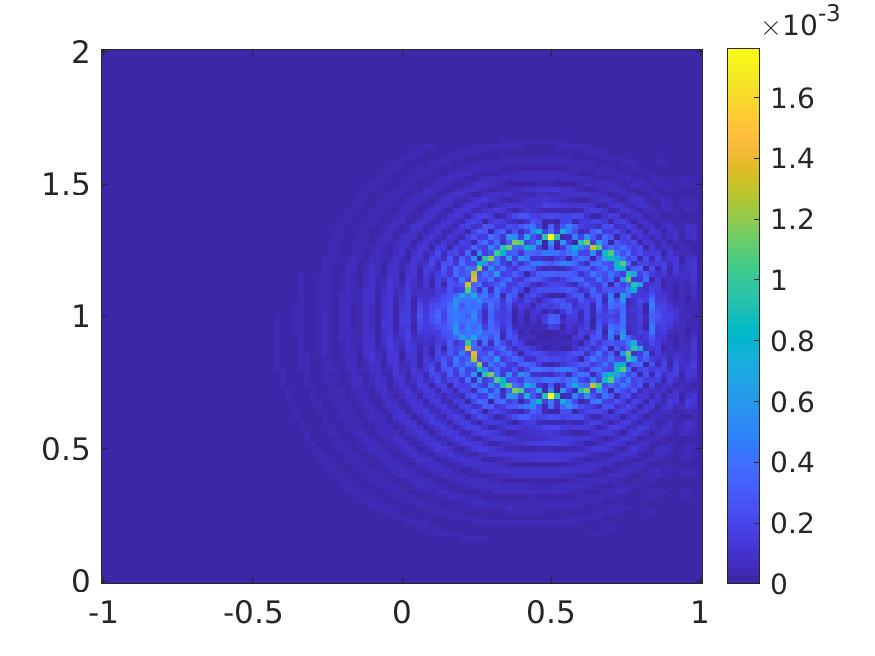}}
	\caption{\label{fig choose N} The function $\Big|u^*(\x, 0) - \ds\sum_{n = 1}^N u^*_n(\x)\Psi_n(0)\Big|$, $\x \in \Omega$ where the function $u^*$ is the solution to \eqref{1.1} and the source is given in Example 2.
	It is evident that the larger $N$, the better approximation in \eqref{2.2222} is.} 
\end{center}
\end{figure}

{\it Step \ref{step 1} in Algorithm \ref{alg 1}.} In this step, we apply the quasi-reversibility method to solve the system \eqref{2.5} with Cauchy data \eqref{data Neu}. 
That means, we minimize the functional $J$ defined in \eqref{2.9}.
The finite difference version of $J$, still called $J$, is
\begin{multline}
	J(u_1, \dots, u_N) = h_\x^2\sum_{m = 1}^N \sum_{i,  j = 2}^{N_\x-1} \Big|\Delta  u_m(x_i, y_j)
	+ \bb(x_i, y_j) \cdot \nabla u_m(x_i, y_j) 
	+ c(x_i, y_j) u_m(x_i, y_j) 
	\\
	- \sum_{n = 1}^N s_{mn} u_n(x_i, y_j)\Big|^2  
	+ h_\x\sum_{m = 1}^N \sum_{j = 1}^{N_\x} (|-\partial_x u_m(x_1, y_j) - \frak f_m(x_1, y_j)|^2 + |\partial_x u_m(x_{N_\x}, y_j) - \frak f_m(x_{N_\x}, y_j)|^2)
	\\
	+h_\x\sum_{m = 1}^N \sum_{i = 2}^{N_\x-1} (|-\partial_y u_m(x_i, y_1) - \frak f_m(x_i, y_1)|^2 + |\partial_y u_m(x_i, y_{N_\x}) - \frak f_m(x_i, y_{N_\x})|^2)
	\\
	+ h_\x \sum_{m = 1}^N \sum_{j = 1}^{N_\x} (|\ u_m(x_1, y_j)|^2 + |\ u_m(x_{N_\x}, y_j)|^2)
	+ h_\x \sum_{m = 1}^N \sum_{i = 2}^{N_\x-1} (|\ u_m(x_i, y_1) |^2 + |\ u_m(x_i, y_{N_\x})|^2)
	\\
	+ \epsilon h_\x^2 \sum_{m = 1}^N \sum_{i = 2}^{N_\x-1} |u_m(x_i, y_j)|^2 + |\nabla u_m(x_i, y_j)|^2 + |\Delta u_m(x_i, y_j)|^2.
	\label{5.1}
\end{multline}
Here, instead of imposing the constraint $u_m = 0$ on $\partial \Omega$, we add additional term:
 $ h_\x \sum_{m = 1}^N \sum_{j = 1}^{N_\x} (|\ u_m(x_1, y_j)|^2 + |\ u_m(x_{N_\x}, y_j)|^2) + h_\x \sum_{m = 1}^N \sum_{i = 2}^{N_\x-1} (|\ u_m(x_i, y_1) |^2 + |\ u_m(x_i, y_{N_\x})|^2)$ to the right hand side of \eqref{5.1}.
This technique significantly reduces the efforts in the implementation. 
We now identify the vector value function $(u_1, \dots, u_N)$ with it ``line up" version
\[
	\frak u_{\frak i} = u_m(x_i, y_j) \quad \mbox{where} \quad \frak i = (i - 1)NN_\x + (j - 1)N + m
\]
for $1 \leq i, j \leq N_\x$ and $1 \leq m \leq N$. 
The data $\frak{f}$ is also line-up in the same manner.
\[
	{\bf f}_{\frak i} = 	\frak f_m(x_i, y_j) \quad \mbox{where} \quad \frak i = (i - 1)NN_\x + (j - 1)N + m
\]
for $i \in \{1, N_\x\}$, $1 \leq j \leq N_\x$ or $1 \leq i \leq N_\x$, $j \in \{1, N_\x\}.$
It is not hard to rewrite $J$ in term of $\frak u$ as
\begin{equation}
	J(\frak u) = |\mathcal L \frak u|^2 + |\mathcal N \frak u - {\bf f}|^2 + |\mathcal D \frak u|^2 + \epsilon|(|\frak u|^2 + |D_x \frak u|^2 + |\frak u|^2 + |L_1 \frak u|^2) 
	\label{J FD}
\end{equation}   for some matrices $\mathcal L$, $\mathcal N$, $\mathcal D$, $D_x$, $D_y$ and $L_1$.
The matrix $\mathcal L$ is such that
\begin{equation*}
	(\mathcal L \frak u)_{\frak i} =  h_\x^2\Big(\Delta u_m(x_i, y_j)
	+ \bb(x_i, y_j) \cdot \nabla u_m(x_i, y_j) 
	+ c(x_i, y_j) u_m(x_i, y_j) 
	- \sum_{n = 1}^N s_{mn} u_n(x_i, y_j)\Big)
\end{equation*}
with $\frak i = (i - 1)NN_\x + (j - 1)N + m,$ $2 \leq i, j \leq N_\x - 1$ and $1 \leq m \leq N.$
The matrix $\mathcal N$ and $\mathcal D$ are the matrices such that $\mathcal N \frak u$ and $\mathcal D \frak u$ respectively correspond to the Neumann and Dirichlet values of $u_m(x_i, y_j)$ where $(x_i, y_j)$ is on $\partial \Omega$, $1 \leq m \leq N.$ The matrix $D_x$, $D_y$ and $L_1$ are such that $D_x \frak u$, $D_y \frak u$ and $L_1 \frak u$ correspond to $\partial_x u_m(x_i, y_j)$, $\partial_y u_m(x_i, y_j)$ and $\Delta u_m(x_i, y_j)$, $2 \leq i, j \leq N_\x-1$, $1 \leq m \leq N.$
The explicit forms of these matrices can be written similarly to \cite[Section 5.1]{LiNguyen:IPSE2019}. For the brevity, we do not repeat the details here.

Since $\frak u$ is the minimizer of $J$ defined in \eqref{J FD}, $u$ satisfies
\[	
	(\mathcal L^{\rm T} \mathcal L + \mathcal N^{\rm T} \mathcal N + \mathcal D^{\rm T} \mathcal D +\epsilon( \Id + D_x^{\rm T} D_x + D_y^{\rm T} D_y + L_1^{\rm T} L_1) ) \frak u = \mathcal N^{\rm T} {\bf f}
\]
where the superscript ${\rm T}$ indicates the transpose of matrices.
This linear system can be solve by any linear algebra package. We employ the command ``lsqlin" of MATLAB for this purpose. 
In all following examples, the regularization parameter $\epsilon$ is chosen to be $\epsilon = 10^{-12}$.

{\it Step \ref{step 4} of Algorithm \ref{alg 1}}. These steps can be implemented directly since they involve only explicit formulas.

Again, the implementation for solving Problem \ref{pro isp Neu} is similar to that for solving Problem \ref{pro isp Dir}. We do not repeat the full process for this case. For the brevity, we just provide some numerical results.

\subsection{Numerical examples}

We now perform four numerical examples for both Problem \ref{pro isp Dir} and Problem \ref{pro isp Neu}.

{\it Example 1.}
We consider the true source function given by 
\[
	p_1^*(x, y) = \left\{
		\begin{array}{ll}
			1 &\mbox{if } (x - 0.5)^2 + y^2 < 0.3^2,\\
			0 &\mbox{otherwise.}
		\end{array}
	\right.
\]

\begin{figure}[h!]
		\subfloat[The true source function]{\includegraphics[width = .3\textwidth]{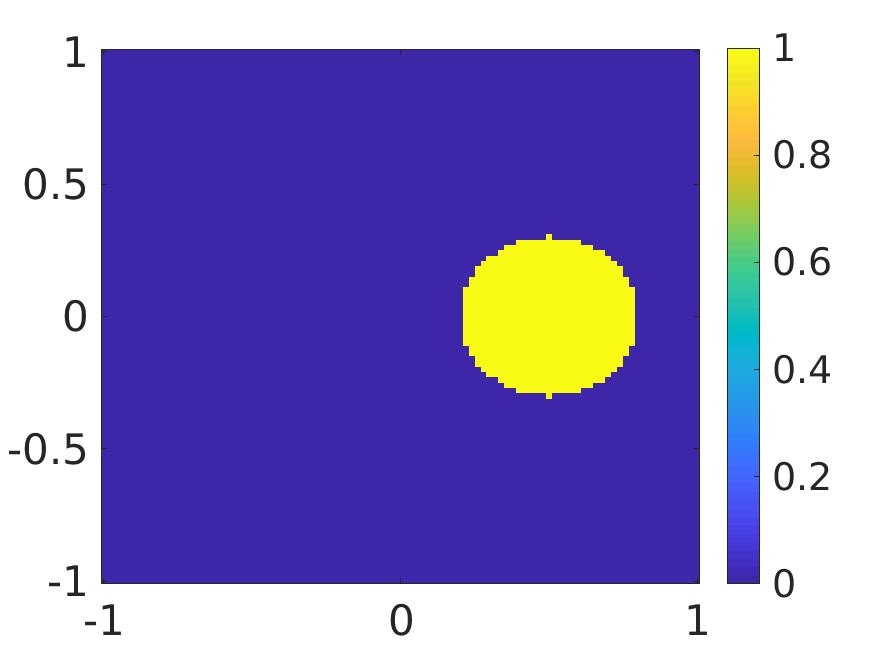}}
		\quad		
		\subfloat[The computed solution to Problem \ref{pro isp Dir} from data with $ 10\%$ noise]{\includegraphics[width = .3\textwidth]{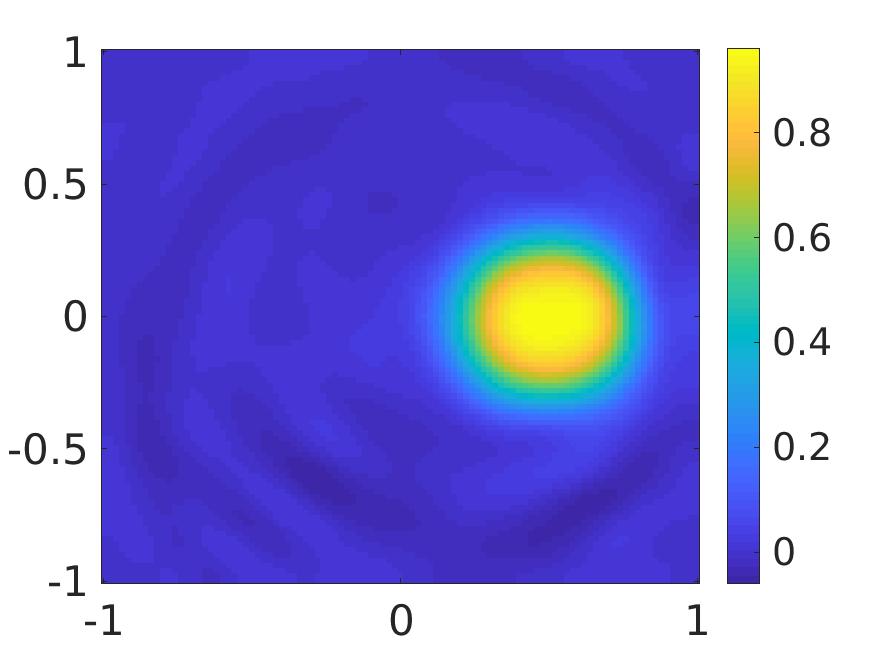}}
		\quad
		\subfloat[The computed solution to Problem \ref{pro isp Dir} from data with $ 100\%$ noise]{\includegraphics[width = .3\textwidth]{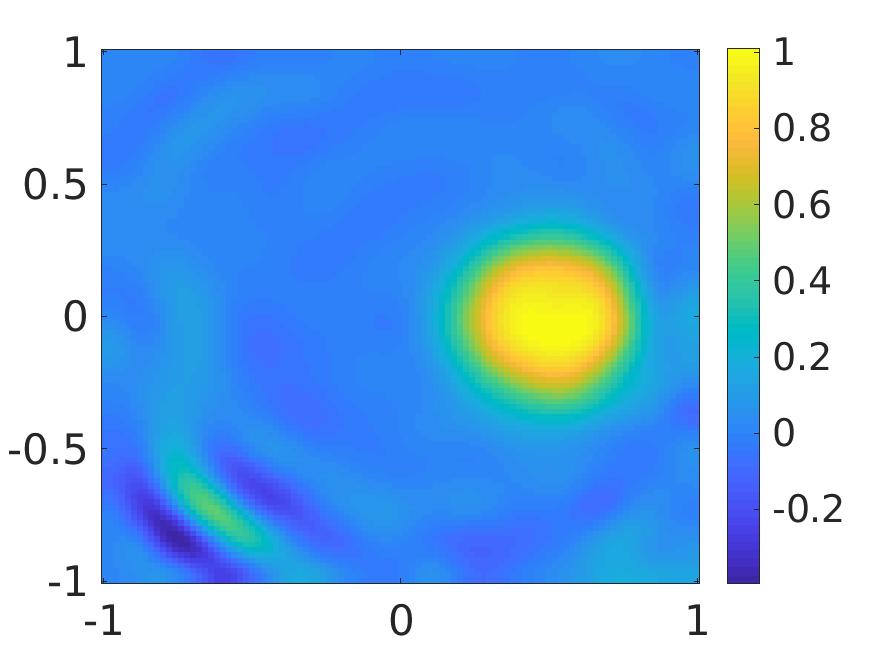}}
		
		\subfloat[The computed solution to Problem \ref{pro isp Neu} from data with $ 10\%$ noise]{\includegraphics[width = .3\textwidth]{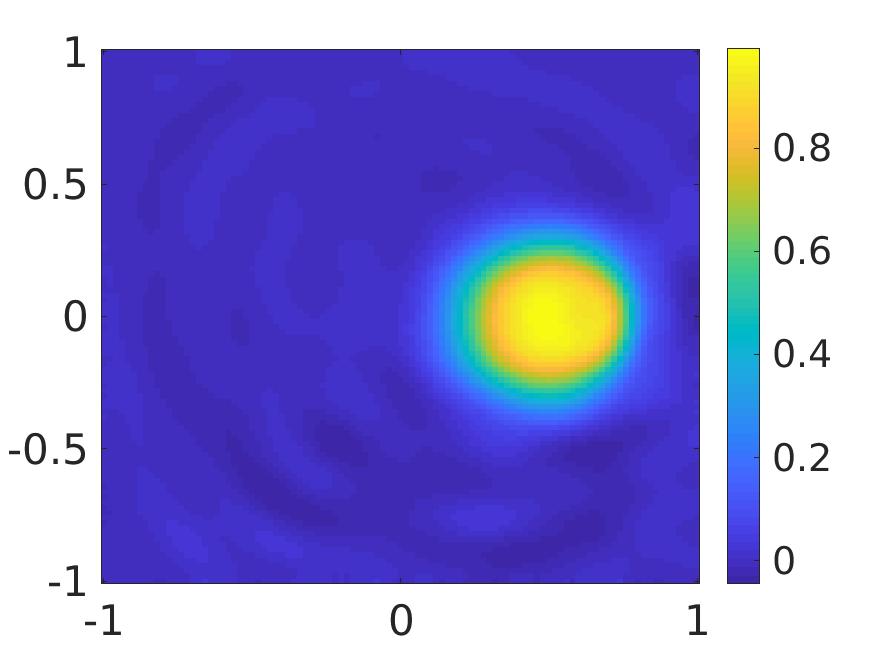}}
		\quad
		\subfloat[The computed solution to Problem \ref{pro isp Neu} from data with $ 100\%$ noise]{\includegraphics[width = .3\textwidth]{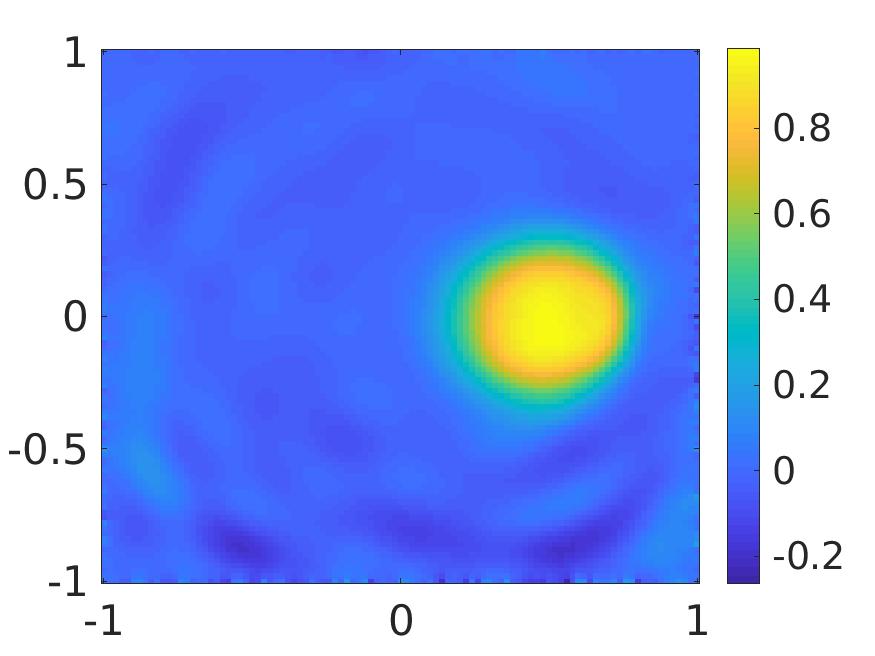}}

		\caption{\label{test 1}Example 1. The true source function and the reconstructions of source functions. }
\end{figure}

The numerical solutions are displayed in Figure \ref{test 1}, which show the accurate reconstructions of the shape and location of the source.
The computed values of the source functions for both Problem \ref{pro isp Dir} and Problem \ref{pro isp Neu} are quite accurate. 
Regarding to Problem \ref{pro isp Dir}, in the case $\delta = 10\%$ the maximal computed value of the source is 0.96115 (relative error 3.9\%) while in the case $\delta = 100\%,$ the maximal computed value of the source is 1.01114 (relative error 1.1\%).
Regarding to Problem \ref{pro isp Neu}, in the case $\delta = 10\%$ the maximal computed value of the source is 0.99389 (relative error 0.6\%) while in the case $\delta = 100\%,$ the maximal computed value of the source is 0.98797 (relative error 1.2\%). 

{\it Example 2}. We consider a more complicated source function
\[
	p_2^*(x, y) = 
	\left\{
		\begin{array}{ll}
			1 &\mbox{if }  (x - 0.5)^2 + y^2 < 0.3^2,\\
			2 &\mbox{if} \, \max\{|x + 0.5|, |y - 0.5|\} < 0.3^2,\\
			0 &\mbox{otherwise},
		\end{array}
	\right.
\]
where the support of the source function consists of a disk and a square, see Figure \ref{test 2} (a).

\begin{figure}[h!]
		\subfloat[The true source function]{\includegraphics[width = .3\textwidth]{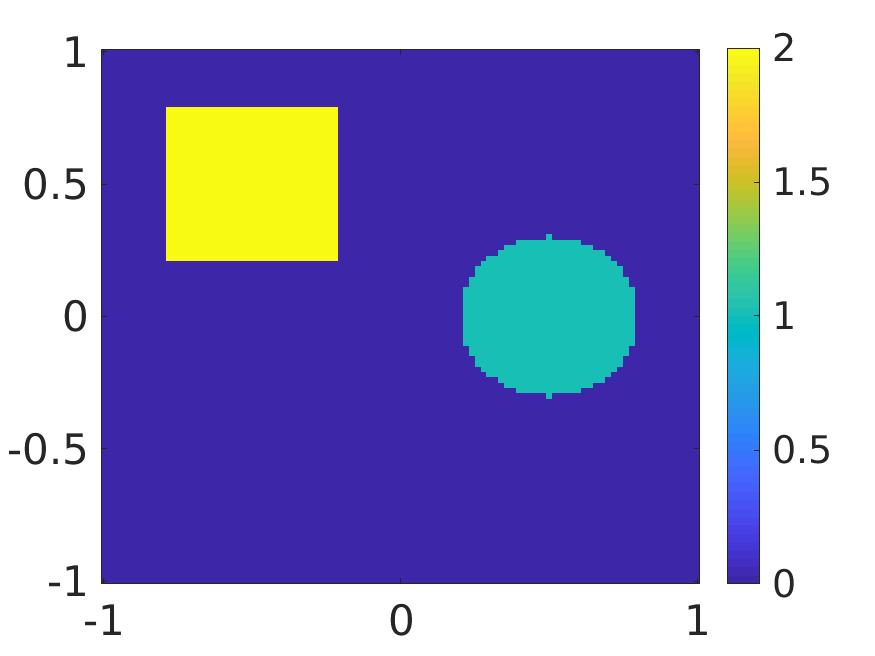}}
		\quad		
		\subfloat[The computed solution to Problem \ref{pro isp Dir} from data with $ 10\%$ noise]{\includegraphics[width = .3\textwidth]{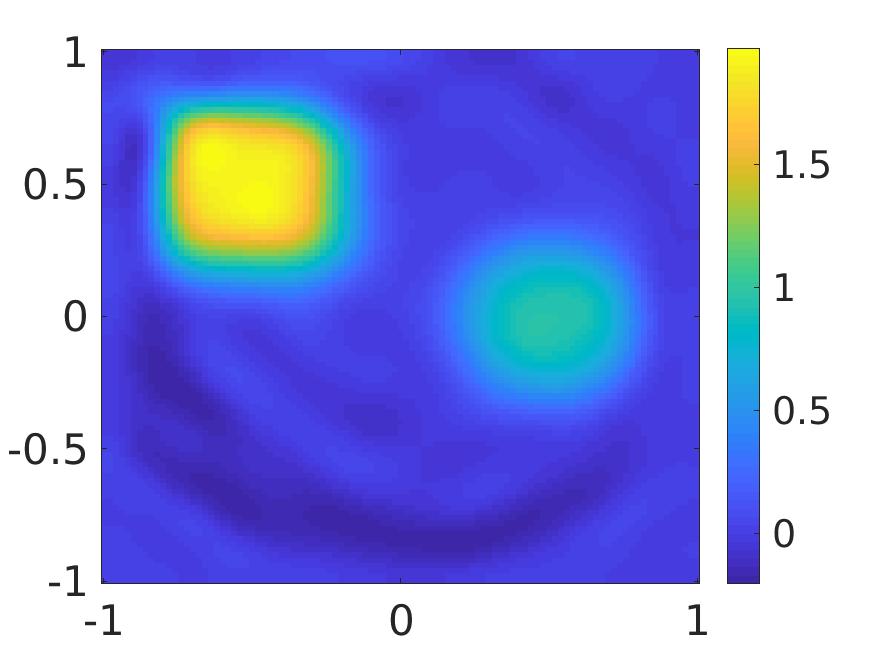}}
		\quad
		\subfloat[The computed solution to Problem \ref{pro isp Dir} from data with $ 100\%$ noise]{\includegraphics[width = .3\textwidth]{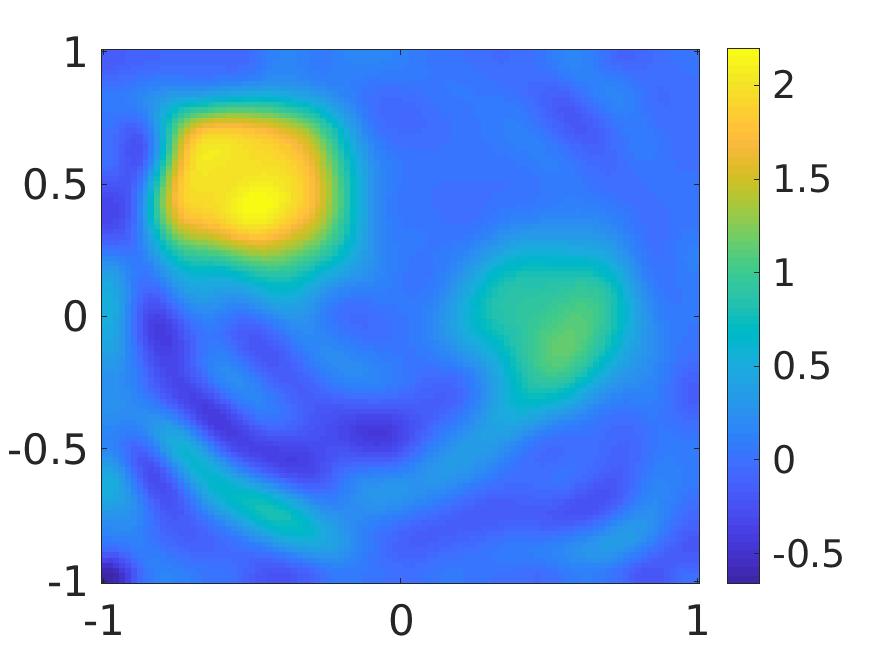}}
		
		\subfloat[The computed solution to Problem \ref{pro isp Neu} from data with $ 10\%$ noise]{\includegraphics[width = .3\textwidth]{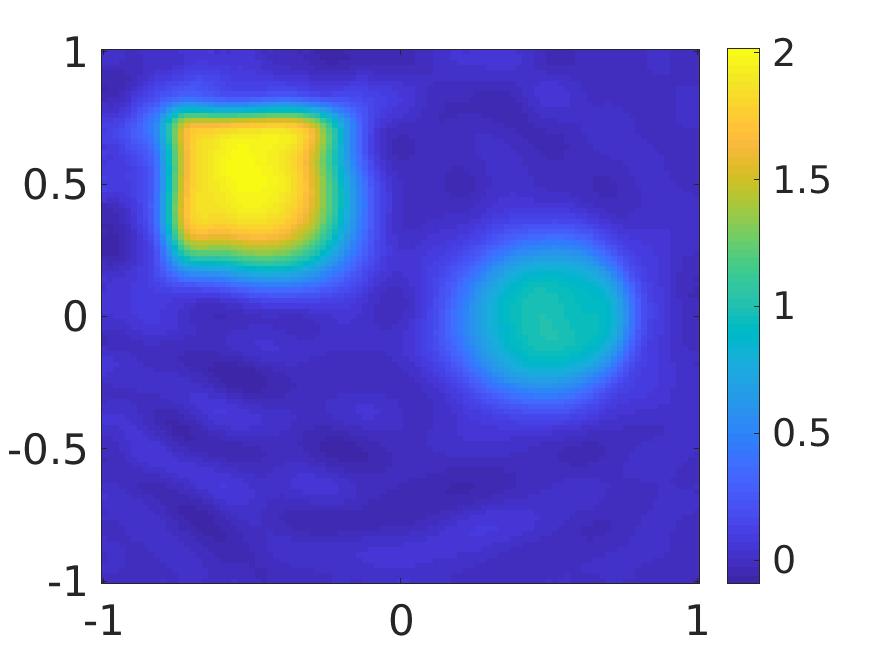}}
		\quad
		\subfloat[The computed solution to Problem \ref{pro isp Neu} from data with $ 100\%$ noise]{\includegraphics[width = .3\textwidth]{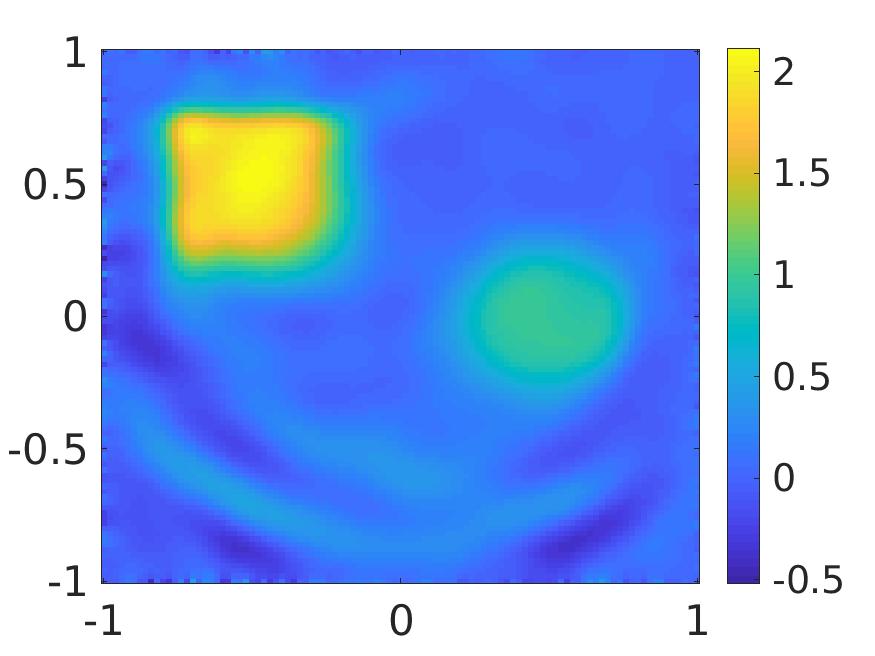}}

		\caption{\label{test 2}Example 2. The true source function and the reconstructions. }
\end{figure}

The reconstructions of source $p_2^{*}$ are displayed in Figure \ref{test 2}, which show the accurate reconstructions of the square and the disk.
The computed values of the source function are quite accurate. 
Regarding to Problem \ref{pro isp Dir}, in the case $\delta = 10\%$ the maximal computed values of the source in the square and the disk are 1.97386 (relative error 1.3\%) and  0.9608 (relative error 3.9\%) respectively 
while in the case $\delta = 100\%,$ the corresponding maximal computed values of the source are 2.19941 (relative error 9.9\%) and 1.157  (relative error 15.7\%) .
Regarding to Problem \ref{pro isp Neu}, in the case $\delta = 10\%$, the maximal computed values of the source in the square and the disk are 2.01843 (relative error 0.9\%) and 0.9994 (relative error 0.0\%) while in the case $\delta = 100\%$, the corresponding maximal computed values of the source are 2.11776 (relative error 5.9\%) and 0.9733 (relative error 2.3\%).
We observe that when the noise level is 100\%, the values of the source are well computed while and the reconstructed shapes of the inclusions start to break out.

{\it Example 3.} We next test the case where the support of the source has more complicated geometries than the one in Example 2, and the source has both positive and negative values. 
\[
	p_3^*(x, t) = \left\{
		\begin{array}{ll}
			3 &\mbox{if } \max\{2|x-0.5|, |y|\} < 0.7 \mbox{ and } (x - 0.5)^2 + y^2 \geq 0.2^2,\\
			-2.5 &\mbox{if } 7(x+0.6)^2 + (y - 0.4)^2 \leq 0.5^2,\\
			0 &\mbox{otherwise}.
		\end{array}
	\right.
\]
The support of the source involves a rectangle with a void and an ellipse, see Figure \ref{test 3}(a).
\begin{figure}[h!]
		\subfloat[The true source function]{\includegraphics[width = .3\textwidth]{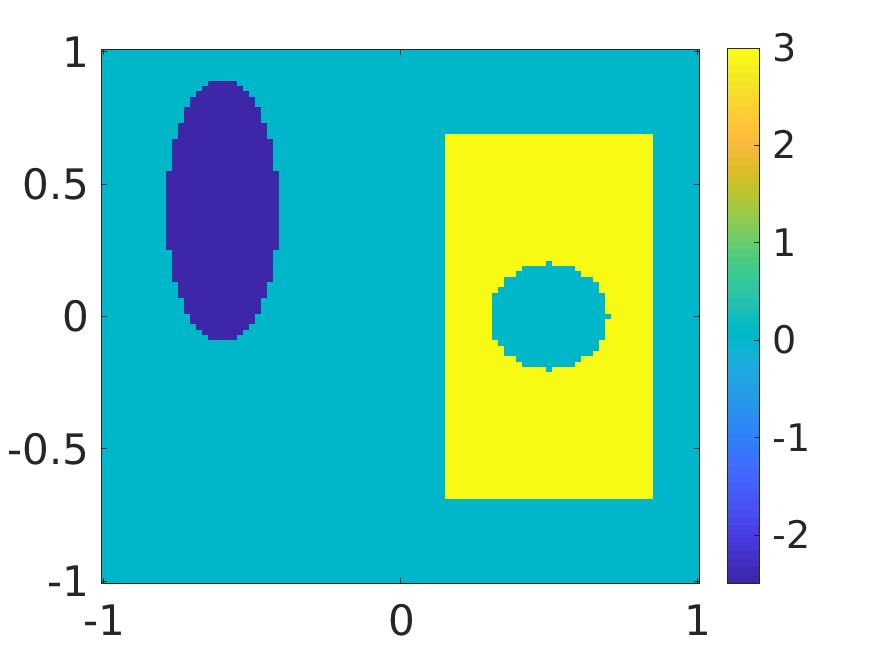}}
		\quad		
		\subfloat[The computed solution to Problem \ref{pro isp Dir} from data with $ 10\%$ noise]{\includegraphics[width = .3\textwidth]{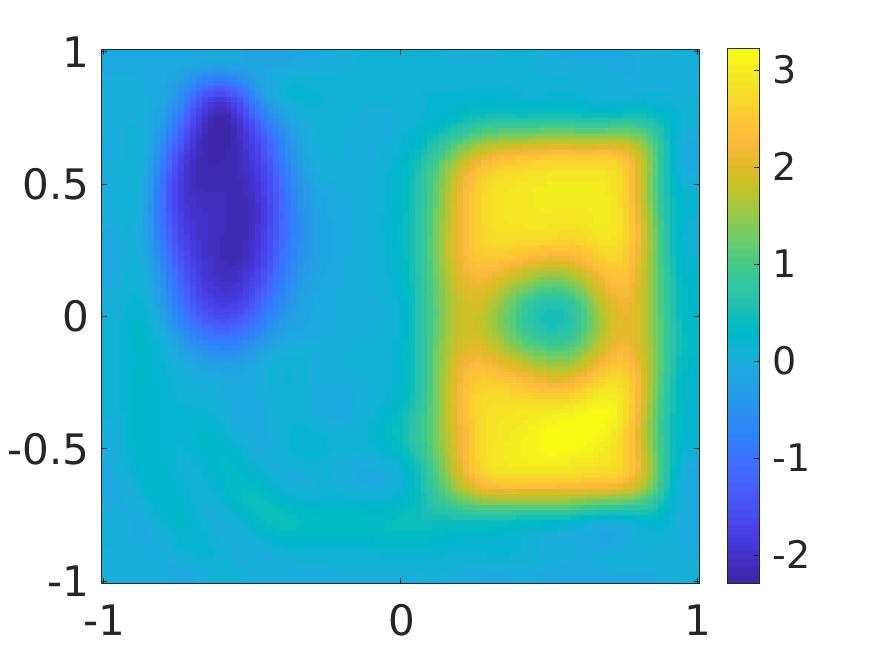}}
		\quad
		\subfloat[The computed solution to Problem \ref{pro isp Dir} from data with $ 100\%$ noise]{\includegraphics[width = .3\textwidth]{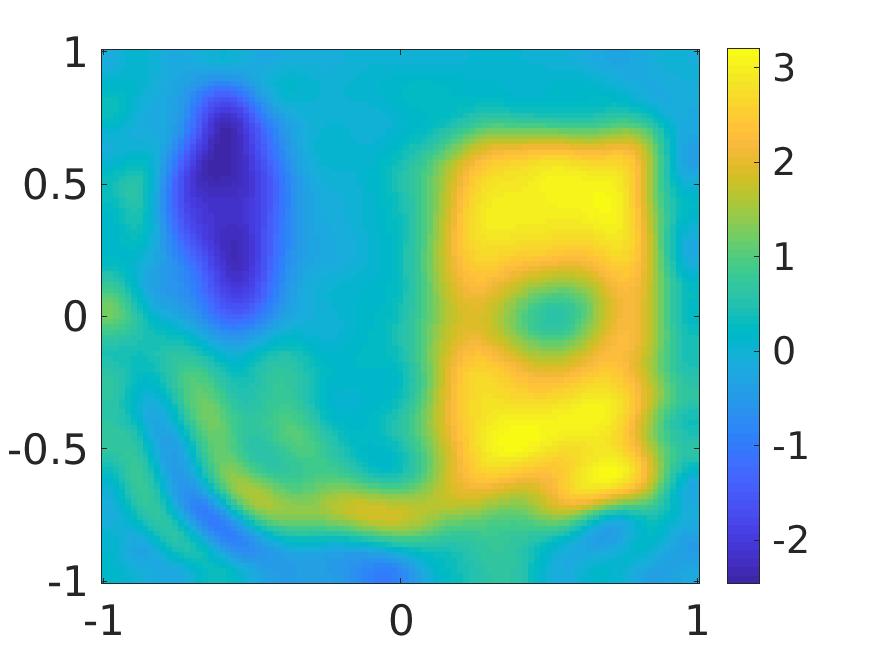}}
		
		\subfloat[The computed solution to Problem \ref{pro isp Neu} from data with $ 10\%$ noise]{\includegraphics[width = .3\textwidth]{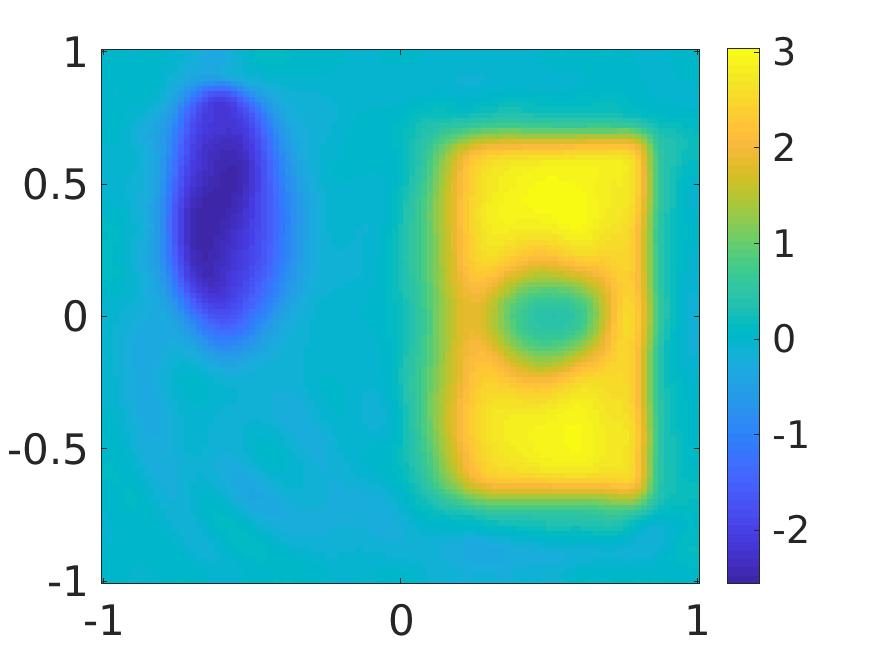}}
		\quad
		\subfloat[The computed solution to Problem \ref{pro isp Neu} from data with $ 100\%$ noise]{\includegraphics[width = .3\textwidth]{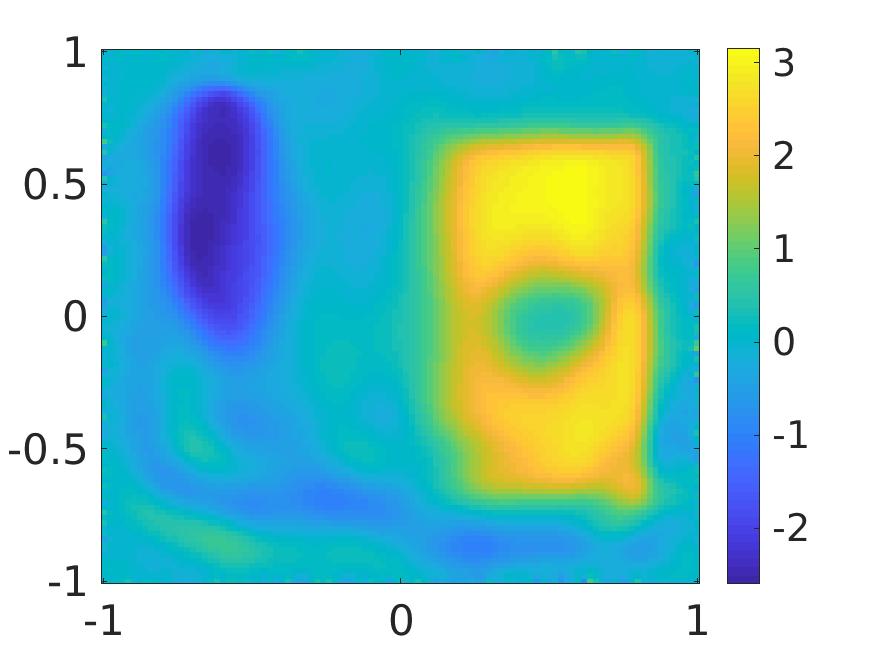}}

		\caption{\label{test 3}Example 3. The true source function and the reconstructions. }
\end{figure}

The numerical solutions of Example 3 are displayed in Figure \ref{test 3}, which show the accurate reconstructions of the rectangle with the void and the ellipse.
The computed values of the source function are quite accurate. 
Regarding to Problem \ref{pro isp Dir}, in the case $\delta = 10\%$ the maximal and minimal computed values of the source is 3.22793 (relative error 7.6\%) and $-2.2951$ (relative error 8.2\%) respectively 
while in the case $\delta = 100\%,$ the maximal and minimal computed values of the source are 3.21003 (relative error 7.0\%) and $-2.4654$ (relative error 1.4\%) respectively.
Regarding to Problem \ref{pro isp Neu}, in the case $\delta = 10\%$, the maximal and minimal computed values of the source are 3.04546 (relative error 1.5\%) and $-2.5617$ (relative error 3.1\%) respectively 
while in the case $\delta = 100\%$, the maximal and minimal computed values of the source are 3.15653 (relative error 5.2\%) and $-2.5978$ (relative error 3.9\%) respectively.
We observe that when the noise level is 100\%, the reconstructed values of the source are almost exact and the shapes of the inclusions are still acceptable. However, some artifacts are present.

{\it Example 4.} In this example, the true source function $p^*$ is the characteristic function of the letter $T$. 

\begin{figure}[h!]
		\subfloat[The true source function]{\includegraphics[width = .3\textwidth]{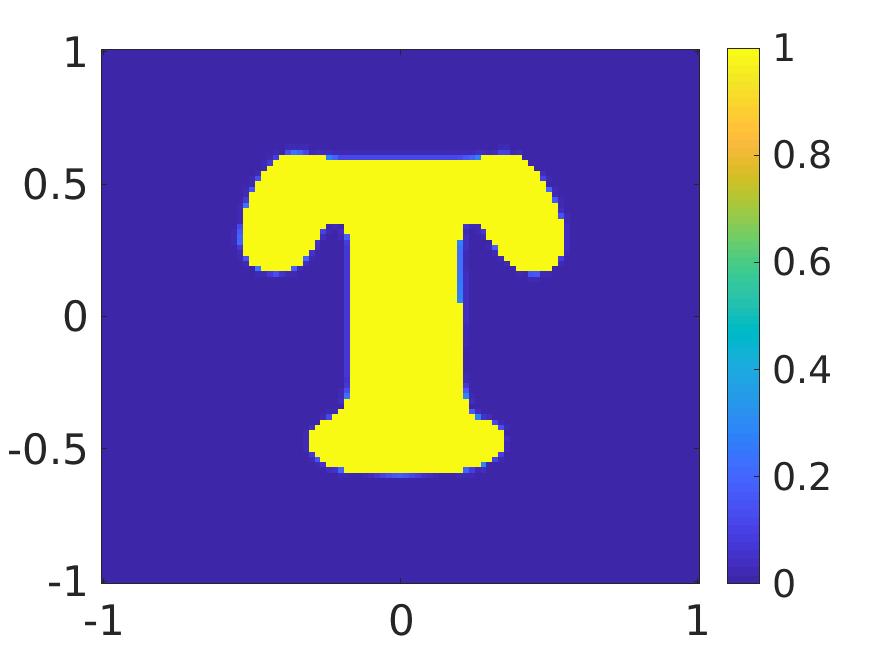}}
		\quad
		\subfloat[The computed solution to Problem \ref{pro isp Dir} from data with $ 10\%$ noise]{\includegraphics[width = .3\textwidth]{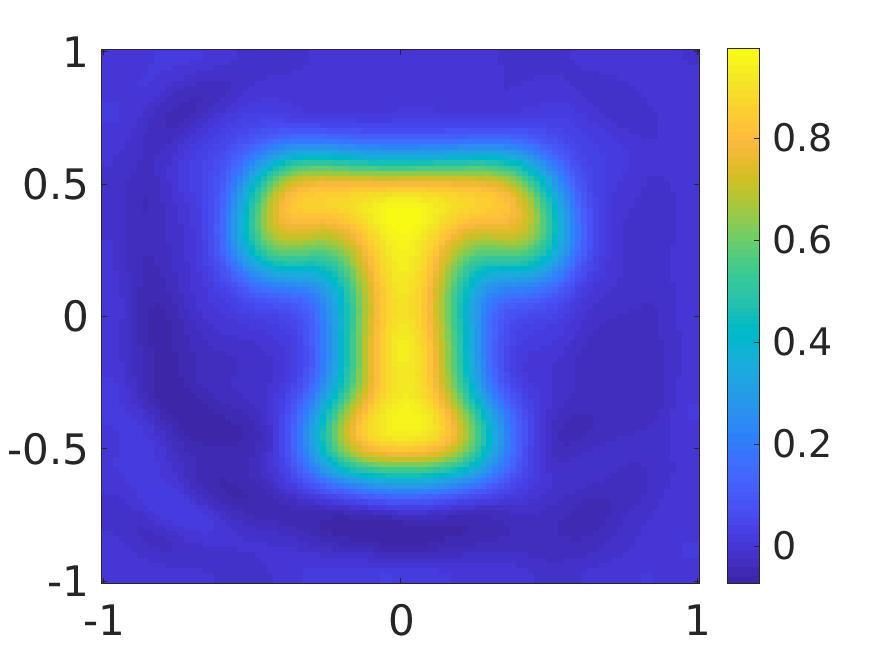}}
		\quad
		\subfloat[The computed solution to Problem \ref{pro isp Dir} from data with $ 100\%$ noise]{\includegraphics[width = .3\textwidth]{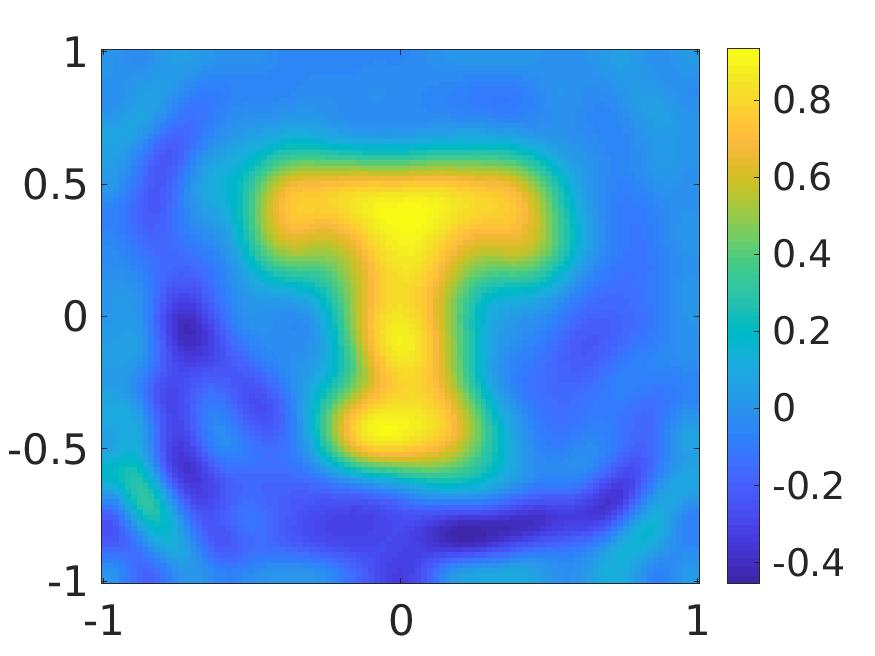}}
		
		\subfloat[The computed solution to Problem \ref{pro isp Neu} from data with $ 10\%$ noise]{\includegraphics[width = .3\textwidth]{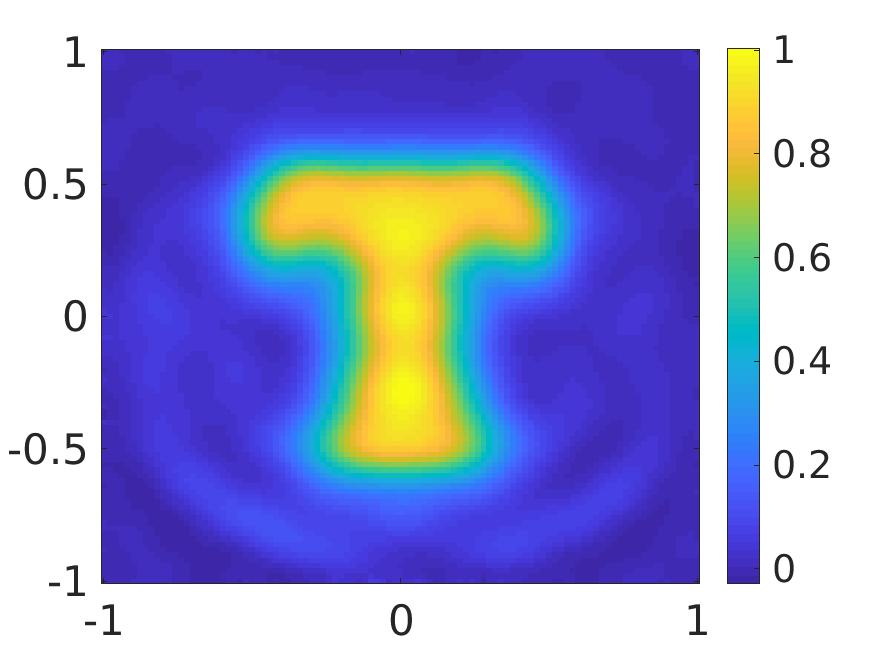}}
		\quad
		\subfloat[The computed solution to Problem \ref{pro isp Neu} from data with $ 100\%$ noise]{\includegraphics[width = .3\textwidth]{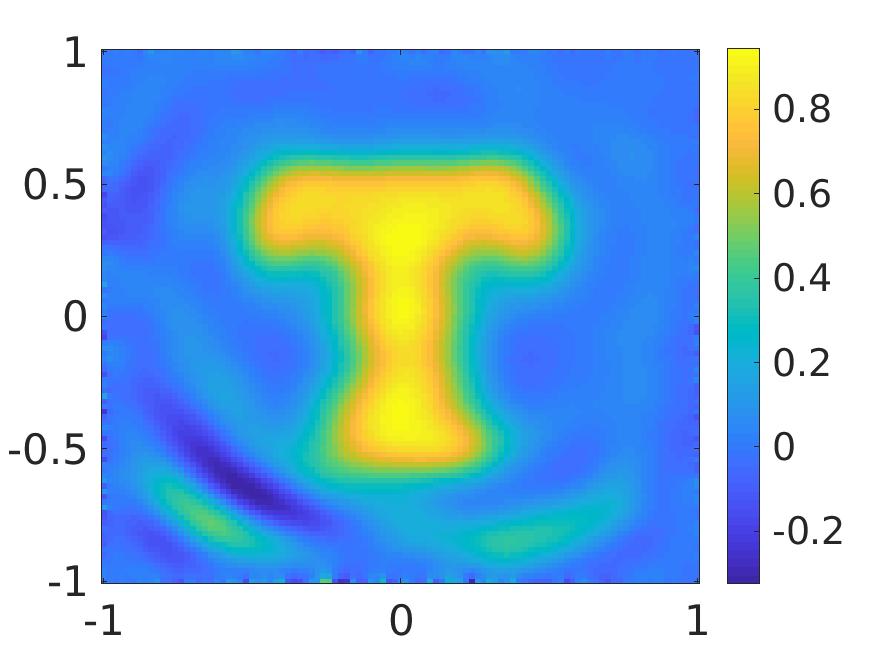}}
		
		\caption{\label{test 4}Example 4. The true source function and the reconstructions. }
\end{figure}

The numerical solutions of Example 4 are displayed in Figure \ref{test 4}, which show the accurate reconstructions of the letter $T$.
The computed values of the source function are quite accurate. 
Regarding to Problem \ref{pro isp Dir}, in the case $\delta = 10\%$ the maximal computed value of the source are 0.97705 (relative error 2.3\%)  while in the case $\delta = 100\%,$ the maximal computed value of the source is 0.93572 (relative error 6.4\%).
Regarding to Problem \ref{pro isp Neu}, in the case $\delta = 10\%$, the maximal computed value of the source is 1.00401 (relative error 0.4\%) and in the case $\delta = 100\%$, the maximal computed value of the source is 0.94551 (relative error 5.4\%).
We observe that when the noise level is 100\%, the reconstructed values and the $T$-shape of the source meet the expectation.

\section{Numerical examples in 1D and the efficiency of the basis $\{\Psi_n\}_{n \geq 1}$}\label{sec Klibanov}

As mention in Remark \ref{basis Psi}, we choose the basis $\{\Psi_n\}_{n \geq 1}$ rather than the the ``sin and cosine" basis of the well-known Fourier series. 
To numerically verifying that this choice is important, we compare some 1D numerical solutions to Problem \ref{pro isp Dir} obtained by our method with respect to two bases: (1) the trigonometric Fourier expansion and (2) the basis $\{\Psi_n\}_{n \geq 1}$. 
We will show that the numerical results in case (1) do not meet the expectation while the numerical results in case (2) do.

For the simplicity, we drop the damping term $a u_t$ in the governing equation.
The governing model is
\begin{equation}
	\left\{
		\begin{array}{rcll}
			\ba(x)u_{tt}(x, t) &=& u_{xx}(x, t) + \bb(x)u_x(x, t) + \bc(x)u(x, t) &(x, t) \in (-1, 1) \times [0, T],\\
			u_t(x, t) &=& 0 &x \in (-1, 1),\\
			u(x, 0) &=& p(x) &x \in (-1, 1),\\
			u(x, t) &=& 0 &x \in \{-1, 1\}.
		\end{array}
	\right.
	\label{main eq}
\end{equation}
Here, we choose $T = 4$. 
The aim of Problem \ref{pro isp Dir} is to compute the initial condition $p(x)$ from the measurement of 
\begin{equation}
	f_1(\pm 1, t) = u_x(\pm 1, t) \quad t \in [0, T].
\end{equation}

We now try the trigonometric Fourier expansion to solve Problem \ref{pro isp Dir}.
In this section, we display the numerical results with the cut-off number $N = 35$. 
We have tried the cases when $N = 50$ and $N = 100$ but the quality of the computed sources does not improve.
For each $x \in (-1, 1)$, we approximate $u(x, t)$ by the $N-$partial sum of its Fourier series
\begin{equation}
	u(x, t) = u_0(x) + \sum_{n = 1}^N u_n(x) \cos \Big(\frac{2\pi n t}{T}\Big) +  	\sum_{n = 1}^N v_n(x)\sin \Big(\frac{2\pi n t}{T}\Big)
	\label{Fourier}
\end{equation}
where 
\[	
	\left\{
		\begin{array}{ll}
		u_0(x) =\ds \frac{1}{T} \int_0^T u(x, t) dt,\\
		u_n(x) = \ds \frac{2}{T}\int_0^T u(x, t)\cos\Big(\frac{2\pi n t}{T}\Big) & n \geq 1,\\
		v_n(x) = \ds \frac{2}{T}\int_0^T u(x, t)\sin\Big(\frac{2\pi n t}{T}\Big) & n \geq 1,\\
		\end{array}
	\right.
	\quad \mbox{for } x \in (-1, 1).
\]

Differentiate \eqref{Fourier} with respect to $t$, we have
\begin{equation}
	u_{tt}(x, t) = -\sum_{n = 1}^N \Big(\frac{2\pi n}{T}\Big)^2u_n(x) \cos \Big(\frac{2\pi n t}{T}\Big)  -  	\sum_{n = 1}^N \Big(\frac{2\pi n}{T}\Big)^2 v_n(x)\sin \Big(\frac{2\pi n t}{T}\Big)
	\label{utt}
\end{equation}

Plugging \eqref{Fourier} and \eqref{utt} into \eqref{main eq} gives
\begin{multline}
	-\ba(x)\Big[\sum_{n = 1}^N \Big(\frac{2\pi n}{T}\Big)^2u_n(x) \cos \Big(\frac{2\pi n t}{T}\Big)  +  	\sum_{n = 1}^N \Big(\frac{2\pi n}{T}\Big)^2 v_n(x)\sin \Big(\frac{2\pi n t}{T}\Big)\Big]
	\\
	=  \Big[
	u_0''(x) + \sum_{n = 1}^N u_n''(x) \cos \Big(\frac{2\pi n t}{T}\Big) +  	\sum_{n = 1}^N v_n''(x)\sin \Big(\frac{2\pi n t}{T}\Big) 
	\Big]
	\\
		+ \bb(x)\Big[
	u_0'(x) + \sum_{n = 1}^N a_n'(x) \cos \Big(\frac{2\pi n t}{T}\Big) +  	\sum_{n = 1}^N v_n'(x)\sin \Big(\frac{2\pi n t}{T}\Big) 
	\Big]
	\\
	+ \bc(x)\Big[
	u_0(x) + \sum_{n = 1}^N u_n(x) \cos \Big(\frac{2\pi n t}{T}\Big) +  	\sum_{n = 1}^N v_n(x)\sin \Big(\frac{2\pi n t}{T}\Big) 
	\Big]
	\label{6.4}
\end{multline}
for all $x \in (-1, 1),$ $t \in [0, T].$
Hence, we have
\begin{equation}
	\begin{array}{ll}
		u_0''(x) + \bb(x) u'_0(x) + \bc(x) u_0(x) = 0,\\
		\ds u_n''(x) + \bb(x) u'_n(x) + \Big[\Big(\frac{2\pi n}{T}\Big)^2a(x)  + \bc(x)\Big]u_n(x) = 0 &n\geq 1\\
		\ds v_n''(x) + \bb(x) v'_n(x) + \Big[\Big(\frac{2\pi n}{T}\Big)^2a(x)  + \bc(x)\Big]v_n(x) = 0 &n\geq 1
	\end{array}
	\label{6.6}
\end{equation}
for $x \in (-1, 1).$
The boundary constraints of $u_0, \{u_n\}_{n \geq 1}$ and $\{v_n\}_{n \geq 1}$ are  
\begin{equation}
	\begin{array}{ll}
		u_0(\pm 1) =  u_n(\pm 1) = v_n(\pm 1) =  0 &n\geq 1,\\
		\ds (u_0)_x(\pm 1) = \frac{1}{T} \int_0^T f_1(\pm 1, t)dt,\\
		\ds (u_n)_x(\pm 1) = \frac{2}{T} \int_0^T f_1(\pm 1, t)\cos\Big(\frac{2\pi t}{T}\Big)  dt &n\geq 1,\\
		\ds (v_n)_x(\pm 1) = \frac{2}{T} \int_0^T f_1(\pm 1, t)\sin\Big(\frac{2\pi t}{T}\Big)  dt &n\geq 1
	\end{array}
	\label{6.7}
\end{equation}
We solve \eqref{6.6}--\eqref{6.7} for $u_0, u_n, v_n$ for $n \geq 1$.
Then, the function initial condition is given by $u(x, 0)$ where $u(x, t)$ is given by \eqref{Fourier}.
We skip presenting the implementation for this method. 
The implementation is similar but simpler than the implementation for the 2D case.

In the numerical tests, we set $\ba(x) = 1 + \sin^2(x^2)$, $\bb(x) = \sin(\pi x)$ and $\bc(x) = \cos(2\pi x)$ for $x \in (-1, 1).$
In this section, we perform three (3) tests. The source functions $p_1$, $p_2$ and $p_3$ correspond to Test 1, Test 2 and Test 3 are given below:
\[
\left.
	\begin{array}{ll}
	p_{1} = \left\{
		\begin{array}{ll}
			\exp(|x - 0.2|^2/(|x - 0.2|^2 - 0.3^2)) &|x - 0.2| < 0.3,\\
			0 &\mbox{otherwise}.
		\end{array}
			\right. \\[2.5ex]
	p_{2} = 1 - x^2, \\[1.5ex]
	p_{3} = \sin(\pi x^3)
	\end{array}
\right.
\]

We compute these source functions with noise level 5\%.
The numerical examples are displayed in Figure \ref{fig 1D}.
It is evident that approximating $u(x, t)$ with the basis $\{\Psi_n\}_{n \geq 1}$ provides much better solutions to Problem \ref{pro isp Dir} in comparison to approximating $u(x, t)$ with the popular ``sin and cosine" basis.
\begin{figure}
\begin{center}
	\subfloat[Test 1]{\includegraphics[width = .3\textwidth]{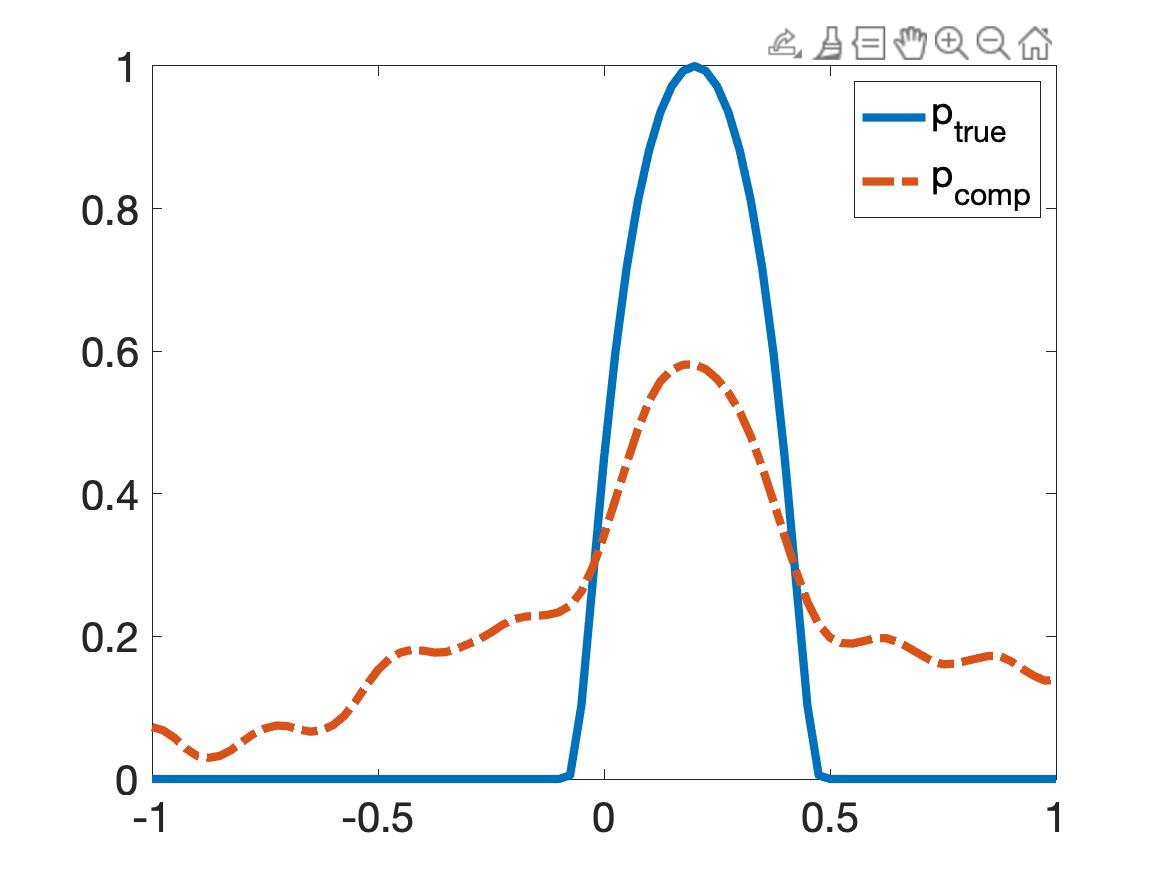}}
	\hfill
	\subfloat[Test 2]{\includegraphics[width = .3\textwidth]{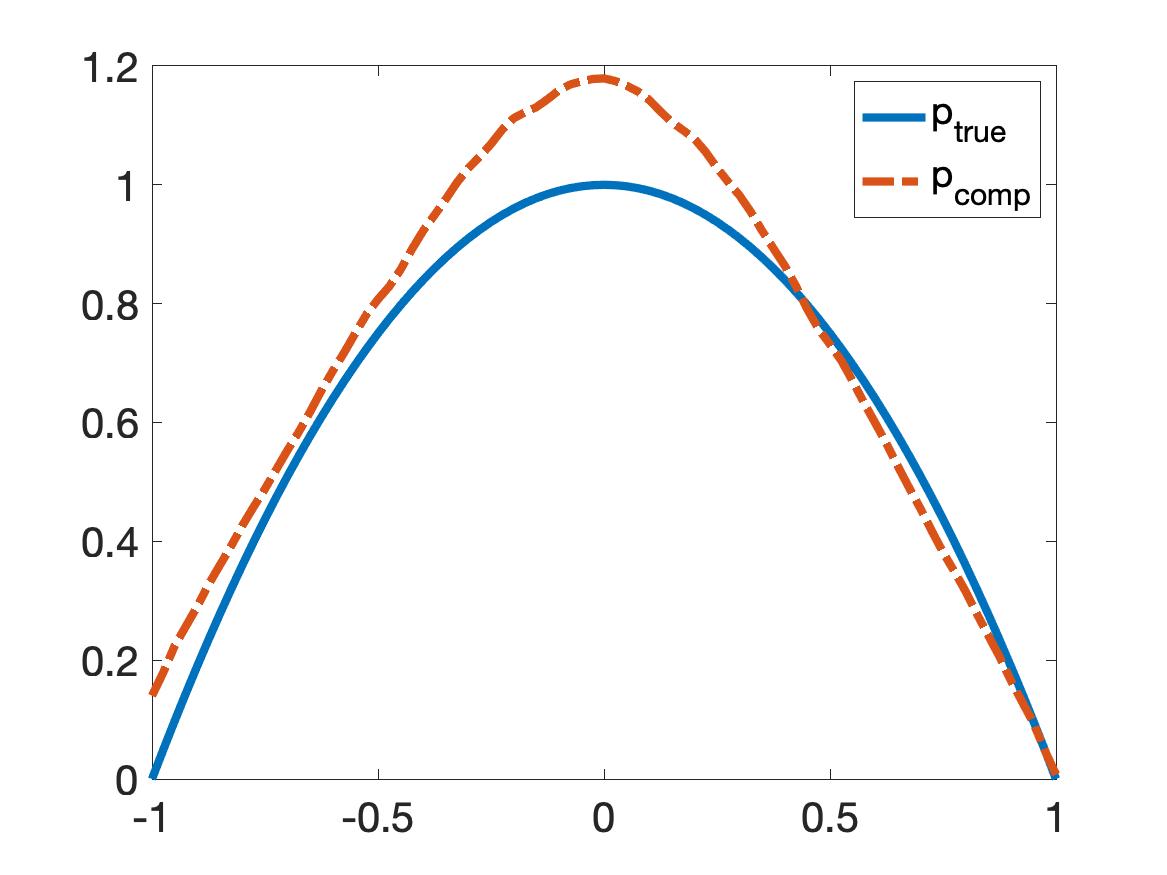}}
	\hfill
	\subfloat[Test 3]{\includegraphics[width = .3\textwidth]{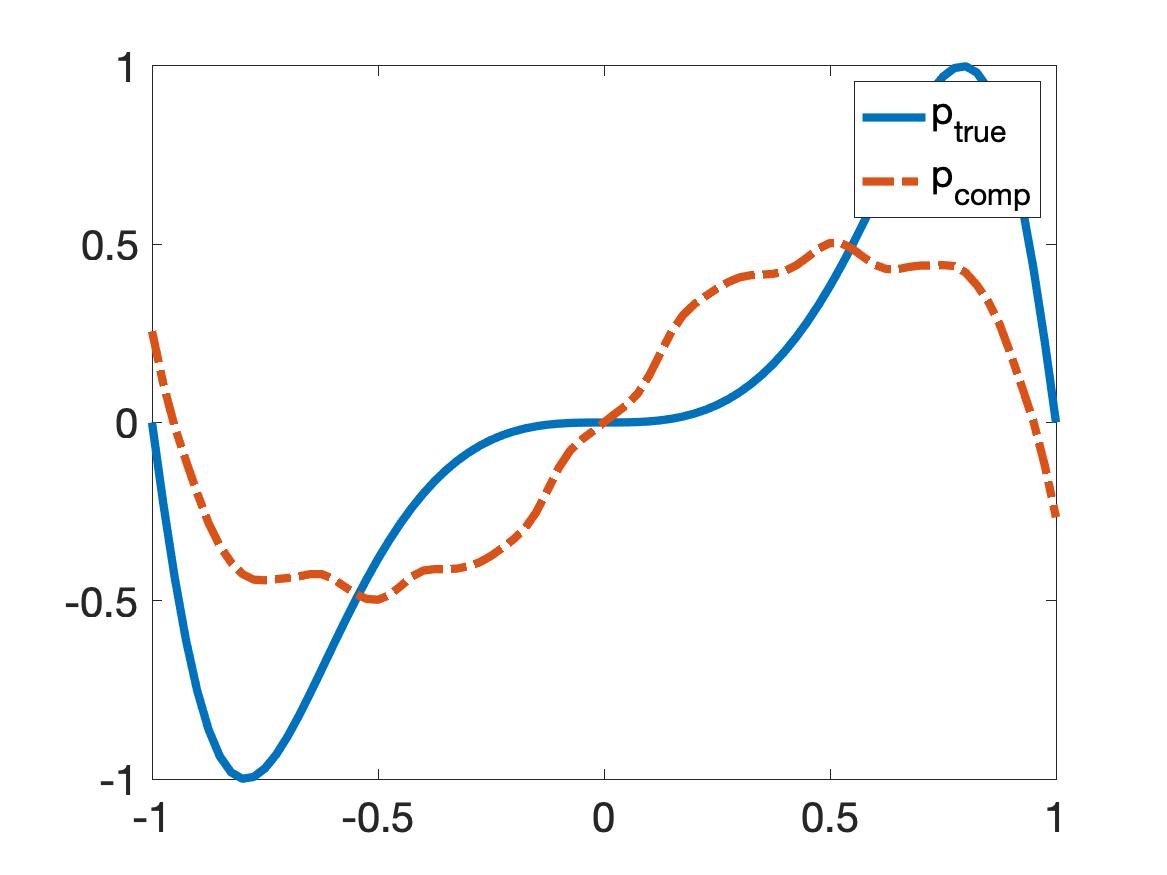}}
	
	\subfloat[Test 1]{\includegraphics[width = .3\textwidth]{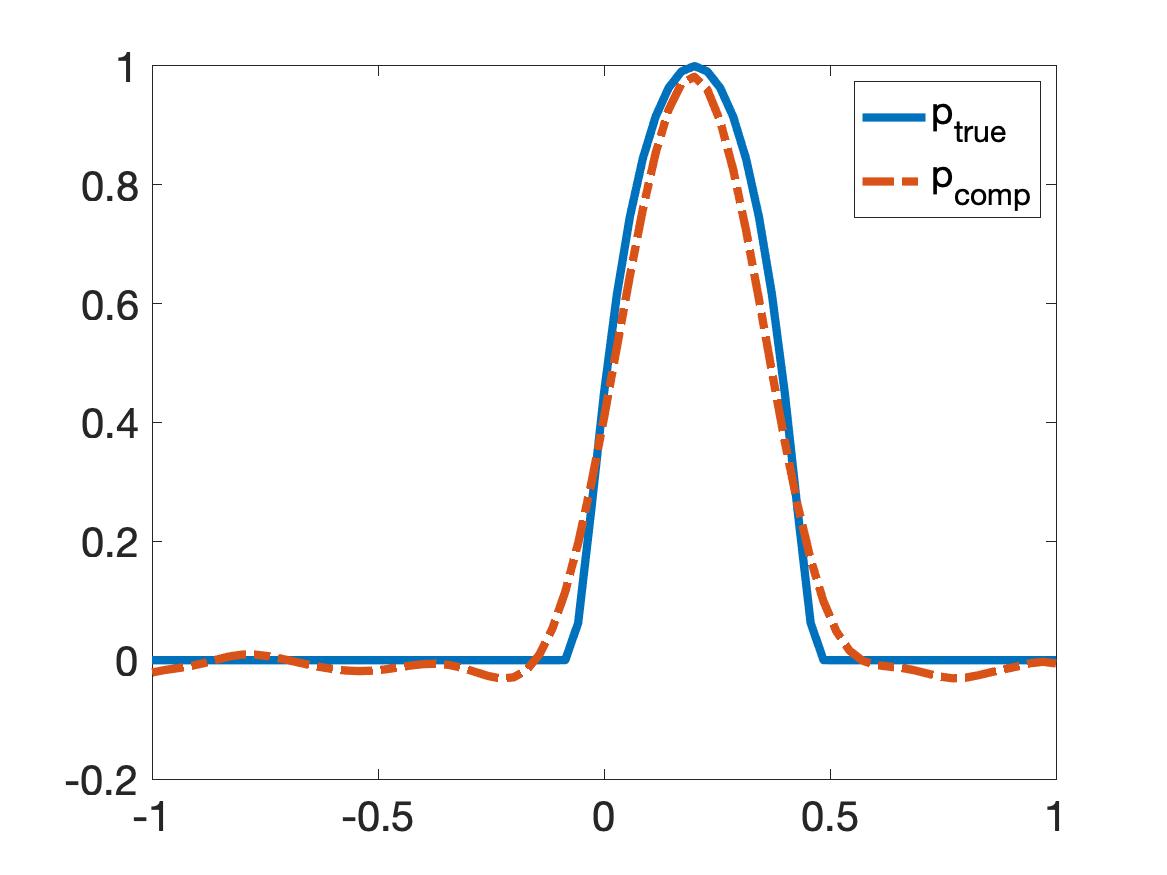}}
	\hfill
	\subfloat[Test 2]{\includegraphics[width = .3\textwidth]{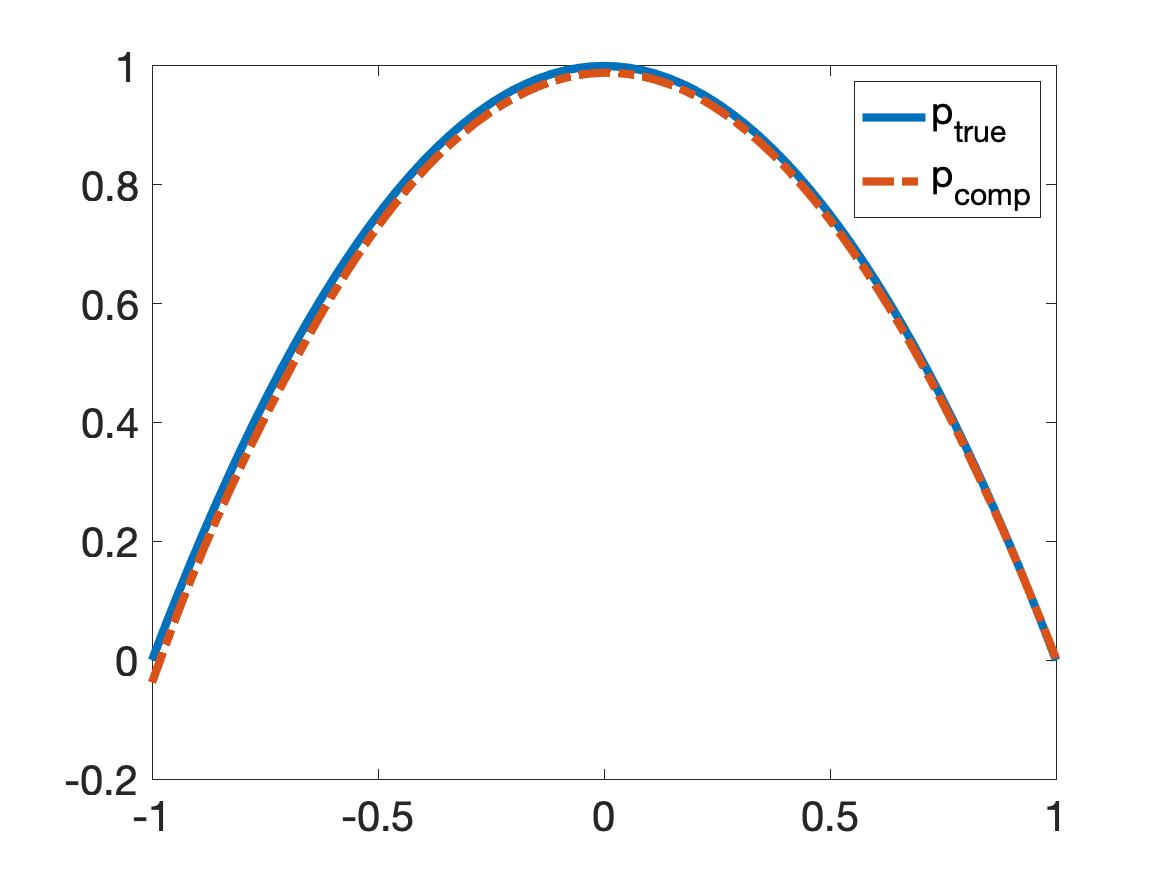}}
	\hfill
	\subfloat[Test 3]{\includegraphics[width = .3\textwidth]{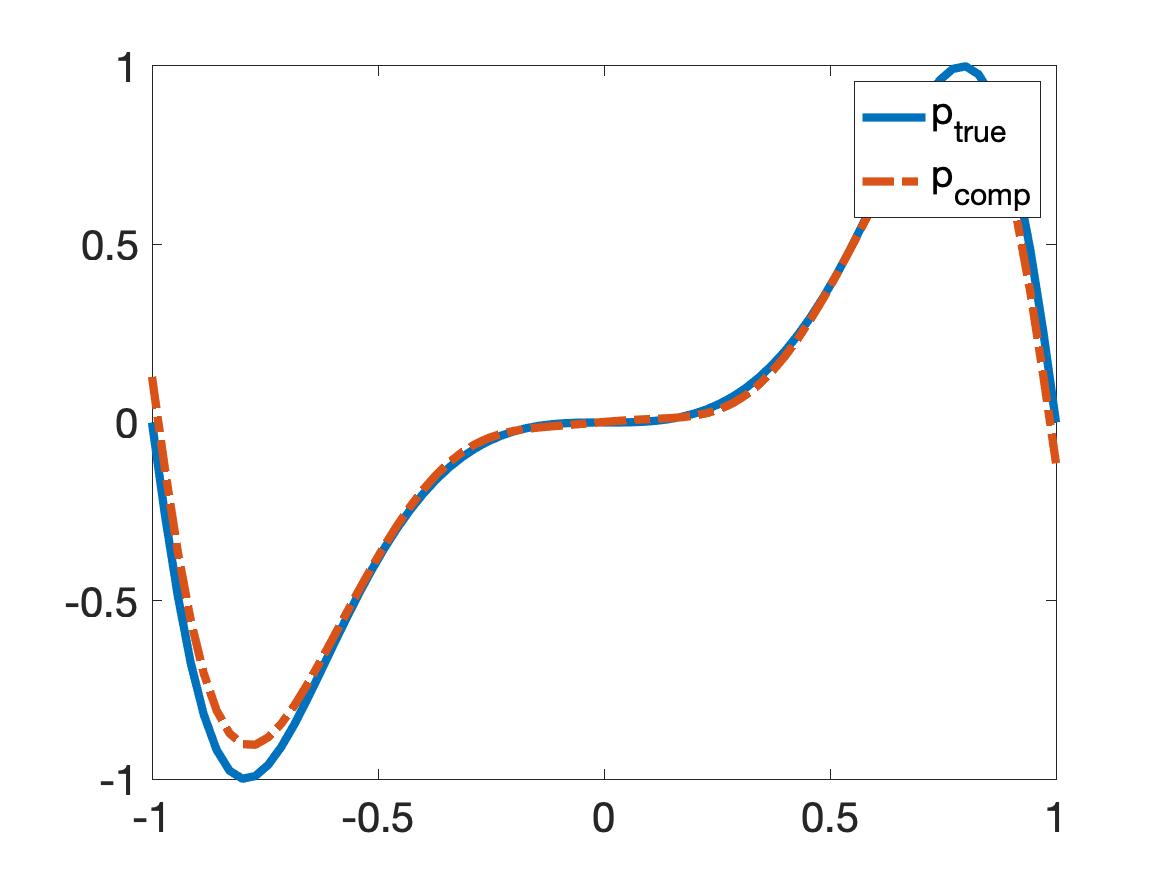}}
	\caption{\label{fig 1D} The true and computed source functions. Row 1 are the results obtained by solving \eqref{6.6}--\eqref{6.7} and row 2 are the results by the 1D version  of Algorithm \ref{alg 1}. 
	It is evident that solving Problem \ref{pro isp Dir} by using the ``sin and cosine" basis is not succesful. In contrast, the reconstructions using the basis $\{\Psi_n\}_{n \geq 1}$ are quite accurate.}
\end{center}
\end{figure}

\section{Concluding remarks} \label{sec con}

In this paper, we introduced a new approach to numerically compute the source function for general hyperbolic equations from the Cauchy boundary data.
In the first step, by truncating the Fourier series of the solution to this hyperbolic equation, we derive an approximation model, who solution directly provides the knowledge of the source. 
We then apply the quasi-reversibility method to solve this system.
The convergence of the quasi-reversibility method is rigorously proved. 
Satisfactory numerical examples illustrates the efficiency of our method.

\section*{Acknowledgments}

Thuy Le and Loc Nguyen are supported by US Army Research Laboratory and US Army Research Office grant W911NF-19-1-0044.
The authors would like to thank Dr. Michael V. Klibanov for many fruitful discussions.

 
\end{document}